\documentclass[11pt,reqno]{amsart}
\usepackage[T1]{fontenc}
\usepackage{amsfonts}
\usepackage{amssymb}
\usepackage{mathrsfs}
\usepackage[latin1]{inputenc}
\usepackage{amsmath}
\usepackage{paralist}
\usepackage[english]{babel}
\usepackage{setspace}

\newtheorem{theorem}{Theorem}[section]
\newtheorem{lemma}[theorem]{Lemma}

\newtheorem{remark}[theorem]{Remark}
\newtheorem{prop}[theorem]{Proposition}

\numberwithin{equation}{section}

\newcommand{\R}{{\mathbb R}}

\newcommand{\Q}{{\mathbb Q}}
\newcommand{\C}{{\mathbb C}}
\newcommand{\N}{{\mathbb N}}

\newcommand{\cL}{{\mathcal L}}

\newcommand{\cB}{{\mathcal B}}

\newcommand{\sC}{\mathsf{C}}

\newcommand{\G}{\Gamma}

\newcommand{\ve}{\varepsilon}

\newcommand{\su}{\subseteq}

\newcommand{\ov}{\overline}

\newcommand\proj{\mathop{\rm proj\,}}

\newcommand\Ker{\mathop{\rm Ker}}

\begin{document}

\title{Mean ergodic properties of the continuous Ces\`aro operators}

\author{Angela\,A. Albanese, Jos\'e Bonet and Werner\,J. Ricker}%A.\,A. Albanese\textsuperscript{*}, J. Bonet\textsuperscript{+} and W.\,J. Ricker}

%\thanks{\textsuperscript{*} Corresponding author}

\address{ Angela A. Albanese\\
Dipartimento di Matematica e Fisica ``E. De Giorgi''\\
Universit\`a del Salento- C.P.193\\
I-73100 Lecce, Italy}
\email{angela.albanese@unisalento.it}

\address{Jos\'e Bonet \\
Instituto Universitario de Matem\'{a}tica Pura y Aplicada IUMPA \\
Universidad Polit\'{e}cnica de Valencia \\
E-46071 Valencia, Spain} \email{jbonet@mat.upv.es}

\address{Werner J.  Ricker \\
Math.-Geogr. Fakult\"{a}t \\
Katholische Universit\"{a}t
Eichst\"att-Ingol\-stadt \\
D-85072 Eichst\"att, Germany}
\email{werner.ricker@ku-eichstaett.de}

\begin{abstract}
Various properties of the (continuous)  Ces\`aro operator $\mathsf{C}$, acting on Banach and Fr\'echet spaces of continuous functions and $L^p$-spaces, are investigated. For instance, the spectrum and point spectrum of $\sC$ are completely determined and a study of certain dynamics of $\sC$ is undertaken (eg. hyper- and supercyclicity, chaotic behaviour). In addition, the mean (and uniform mean) ergodic nature of $\sC$ acting in the various spaces is identified. 
\end{abstract}

\maketitle

%\newpage

\section{Introduction}

\markboth{A.\,A. Albanese, J. Bonet and W.\,J. Ricker}%
{\MakeUppercase{Mean ergodic properties }}

Let $f$ be a $\C$-valued, locally integrable function defined on $\R^+:=[0,\infty)$. 
Then the Ces\`aro average $\sC f$ of $f$ is the function  defined by
\begin{equation}\label{eq.I-1}
\sC f(x):=\frac{1}{x}\int_0^xf(t)\,dt,\quad x\in (0,\infty).
\end{equation}
The linear map $f\mapsto \sC f$ is called the \textit{continuous Ces\`aro operator} (as distinct from the discrete Ces\`aro operator which forms the sequence of averages of vectors coming from various Banach \textit{sequence} spaces) and has been intensively investigated in such Banach spaces as $L^p([0,1])$ and $L^p(\R^+)$, for $1<p<\infty$. The boundedness of $\sC$ on these spaces is due to G.H. Hardy, \cite[p.240]{HLP}, who showed that the operator norm $\|\sC\|_{op}=q$ in both $L^p([0,1])$ and $L^p(\R^+)$, where $\frac{1}{p}+\frac{1}{q}=1$. The spectra and point spectra of $\sC$ are also known; see \cite{B}, \cite{BHS}, \cite{Le}, \cite{Le-2}, for example, and the references therein. Two further Banach spaces on which $\sC$ is naturally defined are the spaces of continuous functions $C([0,1])$ and $C_l([0,\infty])$, both equipped with the sup-norm; here $C_l([0,\infty])$ is the space of all $\C$-valued, continuous functions $f$ on $\R^+$ for which $f(\infty):=\lim_{x\to\infty} f(x)$ exists in $\C$. In both spaces $\|\sC\|_{op}=1$. The spectrum and point spectrum of $\sC$ acting in these spaces are completely determined in Propositions \ref{PS1-C} and \ref{PS2-C}.

The dynamics of $\sC$ have also been investigated in recent years. Recall that a bounded linear operator $T$, defined on a separable Banach space $X$ (or, more generally, a locally convex Hausdorff space $X$, briefly lcHs), is said to be \textit{hypercyclic} if there exists $x\in X$ such that its orbit $\{T^nx\colon n\in\N_0\}$ is dense in $X$. Also, $T$ is called \textit{supercyclic} if, for some $x\in X$, the projective orbit $\{\lambda T^nx\colon \lambda\in \C,\ n\in\N_0\}$ is dense in $X$. Finally, $T$ is said to be \textit{chaotic} if it is hypercyclic and the set of periodic points $\{u\in X\colon \exists\, n\in\N {\rm \ with \ } T^nu=u\}$ is dense in $X$. As general references we refer to \cite{BM}, \cite{GE-P}, for example. It is known that $\sC$ acting on $L^p([0,1])$, $1<p<\infty$, is hypercyclic and chaotic, \cite{LPS}, and that it is not (weakly) supercyclic in $L^2(\R^+)$, \cite{G-LS}. On the other hand, $\sC$ is not supercyclic (hence, not hypercyclic) on $C([0,1])$, \cite{LPS}. We continue an investigation of such properties. For instance, in Proposition \ref{P-C} it is shown that $\sC$ is not supercyclic on $C_l([0,\infty])$ and in Remark \ref{r.(infty)} that $\sC$ has no non--zero periodic points in $L^p(\R^+)$, $1<p<\infty$.

There are also two natural types of \textit{Fr\'echet spaces} in which the Ces\`aro operator $\sC$ acts continuously. One is the Fr\'echet space $C(\R^+)$ consisting of all $\C$--valued, continuous functions on $\R^+$ endowed with the topology of uniform convergence on the compact subsets of $\R^+$. In this space the spectrum of $\sC$ is completely determined and it is shown that $\sC$ is not supercyclic; see Theorem \ref{T.7}. The other class of Fr\'echet spaces consists of the reflexive spaces $L^p_{loc}(\R^+)$, $1<p<\infty$, consisting of all $\C$--valued, measurable functions on  $\R^+$ which are $p$-th power integrable on each set $[0,j]$, for $j\in\N$. In these spaces  the spectrum of $\sC$ is also determined and it is shown that $\sC$ is chaotic (cf. Theorem \ref{T1-Lp}).

%%%%%%%%%%%%%%%%%%%%%%%%%%%%%%%%%%%%%%%%%%%%%%%%%%%%%%%%%%%%%%%%%%%%%%%%%%%%%%%%%%%%%%%%%%%%%%%%%%%%%%%%%%%%%%%%%%%%%%%%%%%%%%%%%%%%%%%%%%%%%%%%%%%%%%%%%%%%%%%%%%%%%%%%%%%%%%%%%%%%%%%%%%%%%%%%%%%%%%
The main point of departure of this paper is actually to investigate (various) mean ergodic properties of $\sC$.
Let   $X$ be a  lcHs and $\Gamma_X$ be a system of continuous seminorms determining the topology of $X$. The strong operator topology $\tau_s$ in the space $\cL(X)$ of all continuous linear operators from $X$ into itself (from $X$ into another lcHs $Y$ we write $\cL(X,Y)$) is determined by the family of seminorms
$q_x(S):=q(Sx)$, for  $S\in\cL(X)$,
for each $x\in X$ and $q\in \Gamma_X$, in which case we write $\cL_s(X)$. Denote by $\cB(X)$ the collection of all bounded subsets of $X$. The topology $\tau_b$ of uniform convergence on bounded sets is defined in $\cL(X)$ via the seminorms
$q_B(S):=\sup_{x\in B}q(Sx)$, for  $S\in\cL(X)$,
for each $B\in \cB(X)$ and $q\in \Gamma_X$; in this case we write $\cL_b(X)$. For $X$ a Banach space, $\tau_b$ is the operator norm topology in $\cL(X)$. If $\Gamma_X$ is countable and $X$ is complete, then $X$ is called a Fr\'echet space. The identity operator on a lcHs $X$ is denoted by $I$. Finally, the \textit{dual operator} of $T\in \cL(X)$ is denoted by $T'\colon X'\to X'$, where $X'=\cL(X,\C)$  is the topological dual space of $X$. As a general reference for lcHs' see \cite{MV}. 

%By $X_\si$ we denote $X$ equipped with its weak topology $\si(X,X')$, where $X'$ is the topological dual space of $X$. The strong topology in $X$ (resp. $X'$) is denoted by $\be(X,X')$ (resp. $\be(X',X)$) and we write $X_\be$ (resp. $X'_\be$); see \cite[IV, Ch. 23]{MV} for the definition. The strong dual space $(X_\be')'_\be$ of $X_\be'$ is denoted simply by $X''$. By $X'_\si$ we denote $X'$ equipped with its weak--star topology $\si(X',X)$. 
%Given $T\in \cL(X)$, its \textit{dual operator} $T'\colon X'\to X'$ is defined by $\langle x, T'x'\rangle=\langle Tx, x'\rangle$ for all $x\in X$, $x'\in X'$. It is known that $T'\in \cL(X'_\si)$ and $T'\in \cL(X_\be')$.%,  \cite[p.134]{KV2}.

The relevant classes of operators are as follows. We say that $T \in \cL(X)$,
with $X$ a lcHs, is \textit{power bounded}  if $\{T^n\}_{n=1}^\infty$ is an equicontinuous subset of $\cL(X)$. For  $X$ a Banach space, this means precisely  that $\sup_{n\in\N}\|T^n\|_{op}<\infty$.  Given $T\in \cL(X)$, we can consider its sequence of averages
\begin{equation}\label{eq.e}
    T_{[n]} := \frac 1 n \sum^n_{m=1} T^m, \qquad      n\in \N,
\end{equation}
called the  Ces\`aro means of $T$.
Then $T$ is called \textit{mean ergodic} (resp., \textit{uniformly mean ergodic}) if  $\{T_{[n]}\}^\infty_{n=1}$
is a convergent sequence in $\cL_s (X)$ (resp., in $\cL_b (X)$).
The study of such operators, initiated by J. von Neumann, N. Dunford, F. Riesz and others, began in the 1930's and has continued ever since; see \cite{K}, \cite[Ch. VIII]{Y} and the references therein. In Theorem \ref{T1-C} it is shown that $\sC$ acting on the Banach space $C([0,1])$ is power bounded and mean ergodic but, fails to be uniformly mean ergodic, whereas in the Banach space $C_l([0,\infty])$ the  Ces\`aro operator $\sC$ is power bounded but, not even mean ergodic (cf. Theorem \ref{T2-C}). In the class  of Banach spaces $L^p([0,1])$, $1<p<\infty$, it is established in Theorem \ref{T1-L} that $\sC$ is neither power bounded nor mean ergodic; the same is shown to be true for $\sC$ acting on $L^p(\R^+)$, $1<p<\infty$ (cf. Theorem \ref{T2-L}). Concerning the above mentioned classes of Fr\'echet spaces in which $\sC$ acts continuously, it is shown in Theorem 	\ref{T.7} that $\sC$ is both power bounded and mean ergodic in $C(\R^+)$ but, not uniformly mean ergodic. Finally, in the Fr\'echet spaces $L^p_{loc}(\R^+)$, $1<p<\infty$, it turns out that $\sC$ is neither power bounded nor mean ergodic (cf. Theorem \ref{T1-Lp}). For recent results on mean ergodic operators in lcHs' we refer to \cite{ABR-1}, \cite{ABR-3}, \cite{ABR-4}, \cite{ABR-7}, \cite{P-1}, \cite{P-2}, for example, and the references therein.

For a Fr\'echet space $X$ and $T\in \cL(X)$, the \textit{resolvent set} $\rho(T)$ of $T$ consists of all $\lambda\in\C$ such that $R(\lambda,T):=(\lambda I- T)^{-1}$ exists in $\cL(X)$. Then $\sigma(T):=\C\setminus \rho(T)$ is called the \textit{spectrum} of $T$. %We will also have the occasions to use the set $\G(T):=\sigma(T)\cap \T$, where $\T:=\{z\in \C\colon |z|=1\}$ is the boundary od $\D:=\{z\in\C\colon |z|<1\}$. 
The \textit{point spectrum} $\sigma_{pt}(T)$ consists of all $\lambda\in\C$ such that $(\lambda I-T)$ is not injective. If we need to stress the space $X$, then we also write $\sigma(T;X)$, $\sigma_{pt}(T;X)$ and $\rho(T;X)$.
Unlike for Banach spaces, it may happen that $\rho(T)=\emptyset$. For example, let $\omega=\C^\N$ be the Fr\'echet space equipped with the lc--topology determined via the seminorms $\{q_n\}_{n=1}^\infty$, where $q_n(x):=\max_{1\leq j\leq n}|x_j|$, for $x=(x_j)_{j=1}^\infty\in \omega$. Then the unit left shift operator $T\colon x\mapsto (x_2,x_3,x_4,\ldots)$, for $x\in \omega$, belongs to $\cL(\omega)$ and, for every $\lambda\in\C$, the element $(1,\lambda,\lambda^2,\lambda^3,\ldots)\in\omega$ is an eigenvector corresponding to $\lambda$.   Or, let $A=\{\alpha_n\colon n\in\N\}$ be \textit{any} countable subset of $\C$ and define $S\in \cL(\omega)$ by $S\colon x\mapsto (\alpha_1x_1,\alpha_2 x_2,\alpha_3 x_3,\ldots)$, for $x\in \omega$. Then $\sigma(S)=\sigma_{pt}(S)=A$ and hence, $\sigma(S)$ need not even be a closed subset of $\C$.

%For a  separable complex Fr\'echet space  $X$ and $T\in \cL(X)$, a vector $x\in X$ is called \textit{cyclic} (resp., \textit{supercyclic, hypercyclic}) for $T$ if the set $\{p(T)x\colon p\ {\rm polynomial}\}$ (resp., $\{\lambda T^nx\colon \lambda\in\C,\ n\in\N_0\}$, $\{T^nx\colon n\in\N_0\}$) is dense in $X$. The operator $T$ is called \textit{supercyclic} (resp., \textit{hypercyclic}) if it has a supercyclic vector (resp., a hypercyclic vector) in $X$.
For ease of reading, some technical (but useful) results which are needed in relation to the spectrum and mean ergodicity of continuous linear operators acting in the class of Fr\'echet spaces called \textit{quojections} (to which $C(\R^+)$ and $L^p_{loc}(\R^+)$, $1<p<\infty$, belong) have been formulated in an Appendix at the end of the paper.

\section{The Ces\`aro operator on Banach spaces of continuous functions}

We consider here the continuous Ces\`aro operator $\sC$ given in \eqref{eq.I-1} when acting
%\[
%Cf(x):=\frac{1}{x}\int_0^xf(t)\, dt,\quad x>0,
%\]
on the Banach spaces $C([0,1])$ and $C_l([0,\infty])$.
% where $C([0,\infty])$ is the space continuous functions on $[0,\infty)$ such that $\lim_{x\to\infty}f(x):=f(\infty)$ exists and is finite.

In order to make the definition of the operator $\sC$ consistent, we set $\sC f(0):=\lim_{x\to 0^+}\sC f(x)=f(0)$ for every $f\in C([0,1])$ or  $f\in C_l([0,\infty])$. It is routine  to check if $f\in C_l([0,\infty])$, then also $\lim_{x\to\infty}\sC f(x)$ exists and equals $f(\infty):=\lim_{x\to\infty}f(x)$, i.e.,  $\sC f(\infty)=f(\infty)$. Then the linear maps $\sC\colon C([0,1])\to C([0,1])$ and  $\sC\colon C_l([0,\infty])\to C_l([0,\infty])$ are well defined with $\|\sC\|_{op}=1$ and satisfy $\sC\textbf{1}=\textbf{1}$, where $\textbf{1}$ is the constant function equal to $1$. Moreover, the null space $\Ker(I-\sC)={\rm span}\{\textbf{1}\}$. Indeed,  every function $f$ satisfying $\sC f=f$ must be continuously differentiable on $(0,1)$ or $(0,\infty)$ via \eqref{eq.I-1}; apply the quotient rule to deduce from \eqref{eq.I-1} and $(\sC f)'=f'$ that $f'\equiv 0$.% Compare with \cite{GF-S}.

In order to investigate the uniform mean ergodicity of  $\sC$  we require  the following two lemmas.

\begin{lemma}\label{L1-C} The closure $\ov{(I-\sC)(C([0,1]))}$ of the range $(I-\sC)(C([0,1]))$ of $(I-\sC)$ is precisely  the space $Z:=\{f\in C([0,1])\colon f(0)=0\}$.
\end{lemma}

\begin{proof} Clearly, $Z$ is a closed subspace of $C([0,1])$. Since $\sC f(0)=f(0)$ for all $f\in C([0,1])$, the space $(I-\sC)(C([0,1]))\su Z$. So, $\ov{(I-\sC)(C([0,1]))}\su Z$. 

For each $n\in\N$, direct calculation yields $(I-\sC)x^n=\left(1-\frac{1}{n+1}\right)x^n$. It follows that
\begin{equation}\label{eq.l-span}
{\rm span}\{x^n\colon n\in\N\}\su (I-\sC)(C([0,1]))\su Z.
\end{equation}
Fix $g\in Z$. By Weierstrass' theorem there exists a sequence of polynomials $\{P_k\}_{k=1}^\infty$ such that $P_k\to g$ uniformly on $[0,1]$. Since $\textbf{1}\not\in Z$, the polynomials $\{P_k\}_{k=1}^\infty$ may not lie in $(I-\sC)(C([0,1]))\su Z$. However, it follows from
  $P_k(0)\to g(0)=0$ that the sequence of polynomials $Q_k(x):=P_k(x)-P_k(0)$, for $x\in [0,1]$ and $k\in\N$, lies in the left-side of \eqref{eq.l-span}. 
	Since also $Q_k\to g$ uniformly on $[0,1]$, we have (via \eqref{eq.l-span}) that  $g\in \ov{{\rm span}\{x^n\colon n\in\N\}}$ and the lemma is proved. 
\end{proof}

\begin{lemma}\label{L2-C} Let $g\in C([0,1])$ belong to $(I-\sC)(C([0,1]))$. Then $g(0)=0$ and, for each $x\in (0,1)$, the limit $\lim_{\ve\to 0^+}\int_{\ve}^x\frac{g(t)}{t}\, dt$ exists.  
\end{lemma}

\begin{proof} Let $f\in C([0,1])$ satisfy $(I-\sC)f=g$. Then $g(0)=0$ and
\[
f(x)-\frac{1}{x}\int_0^xf(t)\, dt=g(x), \quad x\in (0,1].
\]
The function $h:=(f-g)\in C([0,1])$ satisfies $h(x)=\frac{1}{x}\int_0^xf(t)\, dt$, for $x\in (0,1]$, and hence, $h$ is continuously differentiable on $(0,1]$. Since $xh(x)=\int_0^xf(t)\,dt$, for $x\in (0,1]$,  we can conclude via differentiation that $h(x)+xh'(x)=f(x)$, for $x\in (0,1]$. It follows that  $h'(x)=\frac{g(x)}{x}$ on  $ (0,1]$.

Fix $x\in (0,1]$. For each $\ve\in (0,x)$, the continuity of $h'$ on $[\ve,x]$ implies that 
\[
h(x)-h(\ve)=\int_{\ve}^xh'(t)\,dt=\int_{\ve}^x\frac{g(t)}{t}\,dt.
\]
As $h\in C([0,1])$, it follows that $\lim_{\ve\to 0^+}\int_{\ve}^x\frac{g(t)}{t}\, dt=h(x)-h(0)$ exists.
\end{proof}

\begin{theorem}\label{T1-C} The Ces\`aro operator $\sC\colon C([0,1])\to C([0,1])$ is power bounded and  mean ergodic but, not uniformly mean ergodic. Also, $\sC$ fails to be hypercyclic.
\end{theorem}

\begin{proof} Since $\|\sC\|_{op}=1$ and $\sC^n\textbf{1}=\textbf{1}$ for each $n\in\N$, it follows that $\|\sC^n\|_{op}=1$ for each $n\in\N$. Hence, $\sC$ is power bounded; this also implies immediately that  $\sC$ cannot be hypercyclic.

By \cite[Theorem 3]{GF-S}, for every $f\in C([0,1])$, the sequence $\{\sC^nf\}_{n=1}^\infty$ converges to $f(0)\textbf{1}$ in $C([0,1])$. This implies that the operator sequence of iterates $\{\sC^n\}_{n=1}^\infty$ converges to the projection $P\colon C([0,1])\to C([0,1])$ given by  $f\mapsto Pf:=f(0)\textbf{1}$, in $\cL_s(C([0,1]))$. Since the averages of any convergent sequence converge to the same limit, the arithmetic means $\{\sC_{[n]}\}_{n=1}^\infty$ also converge to $P$ in   $\cL_s(C([0,1]))$, i.e., $\sC$ is mean ergodic.

It remains to  show that $\sC\colon C([0,1])\to C([0,1])$ is not uniformly mean ergodic. Since $\|\sC^n\|_{op}=1$, for each $n\in\N$, we have $\lim_{n\to\infty}\frac{\|\sC^n\|_{op}}{n}=0$. A theorem of M. Lin, \cite{Li}, asserts that the uniform mean ergodicity of $\sC$ is then equivalent to the the range $(I-\sC)(C([0,1]))$ of $(I-\sC)$ being closed in $C([0,1])$. 

Consider the continuous function $g(x):=-1/(\log x)$, for $x\in (0,1/2]$, with $g(0):=0$ and $g(x):=1/(\log 2)$, for $x\in [1/2,1]$. By Lemma \ref{L1-C} the function $g\in \ov{(I-\sC)(C([0,1]))}$. On the other hand, for every $\ve\in (0,1/2)$, we have
\[
\int_{\ve}^{1/2}\frac{g(t)}{t}\, dt=-\int_{\ve}^{1/2}\frac{dt}{t\log t}=\log (-\log\ve)-\log (\log 2),
\]
%So, $\int_{\ve}^{1/2}\frac{g(t)}{t}\, dt\to \infty$ 
which tends to $\infty$ as $\ve\to 0^+$. By Lemma \ref{L2-C}  it follows that $g\not\in (I-\sC)(C([0,1]))$, i.e., $(I-\sC)(C([0,1]))$ is not closed in $C([0,1])$. Then  the stated result of M. Lin yields that  $\sC$ is not uniformly mean ergodic. \end{proof}

\bigskip

Our next result is stated (correctly) in \cite[Theorem 2.7]{LPS}. However, the proof given there is incorrect as it is based on the claim that $\sigma_{pt}(\sC')=\emptyset$, which is \textit{not} the case. Indeed, the Dirac point measure $\delta_0$ induces the element of $(C([0,1]))'$ given by
%depends on the Positive Supercyclicity Theorem of \cite{LS-M} that cannot be applied because 
$\delta_0\colon f\mapsto f(0)$, which  satisfies $\sC'\delta_0=\delta_0$. Hence, 
$\sigma_{pt}(\sC')\not =\emptyset$. 

\begin{prop}\label{P.8} The Ces\`aro operator  $\sC\colon C([0,1])\to C([0,1])$ is not supercyclic.
\end{prop}

\begin{proof} Proceeding by contradiction, 
suppose that there exists a supercyclic vector $g\in C([0,1])$ for $\sC$, i.e., the set  $\{\lambda \sC^ng\colon \lambda\in\C, \ n\in\N_0\}$ is dense in $C([0,1])$. Then there exist a sequence $\{\lambda_k\}_{k=1}^\infty\su \C$ and an increasing sequence $\{n_k\}_{k=1}^\infty\su \N$ such that $\lambda_k \sC^{n_k}g\to \textbf{1}$ in $C([0,1])$ for $k\to\infty$. In particular, $\lambda_k  \sC^{n_k}g(0)=\lambda_k g(0)\to 1$ for $k\to\infty$ and so $g(0)\not=0$.
On the other hand,
given any function $f\in C([0,1])$ such that $f\not=0$ but $f(0)=0$ (eg., $f(x)=x$ for $x\in [0,1]$), there exist a sequence $\{\mu_r\}_{r=1}^\infty\su \C$ and an increasing sequence $\{m_r\}_{r=1}^\infty\su \N$ such that $\mu_r \sC^{m_r}g\to f$ in $C([0,1])$ for $r\to\infty$. Hence, $\mu_r  \sC^{m_r}g(0)=\mu_r g(0)\to f(0)=0$ for $r\to\infty$ and so $\mu_r\to 0$ for $r\to\infty$ (as $g(0)\not=0$). Since $\|\mu_r\sC^{m_r}g\|_\infty\leq |\mu_r|\cdot\|g\|_\infty$ for all $r\in\N$,  it follows that  
%$\|\lambda_kC_1^{n_k}g\|_\infty\to 0$ for $k\to\infty$. Therefore,  
$\|f\|_\infty=0$;  a contradiction to $f\not=0$.
\end{proof}

\begin{theorem}\label{T2-C} The Ces\`aro operator  $\sC\colon C_l([0,\infty])\to C_l([0,\infty])$ is power bounded, not hypercyclic and not mean ergodic. Moreover,
\begin{equation}\label{e.range}
\ov{(I-\sC)(C_l([0,\infty]))}=\{f\in C_l([0,\infty])\colon f(0)=f(\infty)=0\}
\end{equation}
\end{theorem}

\begin{proof} For each $n\in\N$, we have $\|\sC^n\|_{op}=1$ and so $\sC$ is power bounded. In particular, $\sC$ is then not hypercyclic.

To prove that $\sC$ is not mean ergodic, select any function $f\in C_l([0,\infty])$ satisfying $f(0)=1$ and $f(\infty)=0$ (eg., $f(x)=(\cos x)/(x+1)$, for $x\in \R^+$). Recall that $\sC^nf(0)=f(0)=1$ and $\sC^nf(\infty)=f(\infty)=0$ for all $n\in\N$.
Observe, for each $j\in\N$, that $\sC^n\textbf{1}=\textbf{1}$ on $[0,j]$ and, for $k\in\N$ fixed, that the sup-norm of $\sC^nx^k$ on $[0,j]$ equals $\frac{j^k}{(k+1)^n}$, for all $n\in\N$, i.e.,  $\sC^nx^k\to 0$ uniformly on $[0,j]$ for $n\to\infty$.
 For each $j\in\N$, let $f_j$ denote the restriction of $f$ to $[0,j]$. The previous observation, together with an examination of the proof of Theorem 3 in \cite{GF-S}, shows that  $\{\sC^nf_j\}_{n=1}^\infty$ converges to the constant function $f(0)\textbf{1}$ uniformly on $[0,j]$. 
Now, suppose that $\sC$ is mean ergodic in $C_l([0,\infty])$.  Then there exists $g\in   C_l([0,\infty])$ such that the sequence $\sC_{[n]}f=\frac{1}{n}\sum_{k=1}^n\sC^kf$ converges to $g$ in $C_l([0,\infty])$ for $n\to\infty$. The functional $u\colon C_l([0,\infty])\to \C$ given by  $u(h):=h(\infty)$ for $h\in C_l([0,\infty])$, is linear and continuous with $|u(h)|\leq \|h\|_\infty$ for all $h\in C_l([0,\infty])$. In particular, $\sC_{[n]}f(\infty)\to g(\infty)$ as $n\to\infty$. But,
\[
\sC_{[n]}f(\infty)=\frac{1}{n}\sum_{k=1}^n\sC^kf(\infty)=\frac{1}{n}\sum_{k=1}^n f(\infty)=f(\infty)=0,
\]
for all $n\in\N$. Accordingly, $g(\infty)=0$. On the other hand, for each $j\in\N$ we have $\sC_{[n]}f_j\to g$ uniformly on $[0,j]$ for $n\to\infty$; this is immediate from $\sC_{[n]}f\to g$ in $C_l([0,\infty])$ for  $n\to\infty$. Since $\sC^nf_j\to \textbf{1}$ uniformly on $[0,j]$ for $n\to\infty$, we have $g=\textbf{1}$ on $[0,j]$. As $j$ is arbitrary, this implies that $g(x)=1$ for all $x\in \R^+$. It follows that $g(\infty)=\lim_{x\to\infty} g(x)=1$; a contradiction. So, $\sC$  is not mean ergodic.

It remains to prove the statement \eqref{e.range}. Clearly, 
\[
Z:=\{f\in C_l([0,\infty])\colon f(0)=f(\infty)=0\}
\]
is closed in $C_l([0,\infty])$ and $(I-\sC)(C_l([0,\infty]))\su Z$. So, $\ov{(I-\sC)(C_l([0,\infty]))}\su Z$.

For each $m,\ n\in\N$, define $h_{m,n}(x)=x^n$ if $x\in [0,m]$ and $h_{m,n}(x)=m^n$ if $x\in [m,\infty)$. Then $(I-\sC)h_{m,n}(x)=\left(1-\frac{1}{n+1}\right)x^n$ if $x\in [0,m]$ and $(I-\sC)h_{m,n}(x)=\left(1-\frac{1}{n+1}\right)\frac{m^{n+1}}{x}$ if $x\in [m,\infty)$. For each $m,\ n\in\N$, let $g_{m,n}:=\frac{n+1}{n}(I-\sC)h_{m,n}$. Then, for every $m,\ n\in\N$, we have $g_{m,n}\in (I-\sC)(C_l([0,\infty]))$ with   $g_{m,n}(x)=x^n$ if $x\in [0,m]$ and $g_{m,n}(x)=\frac{m^{n+1}}{x}$ if $x\in [m,\infty)$. So, ${\rm span}\{g_{m,n}\colon m,\ n\in\N\}\su (I-\sC)(C_l([0,\infty]))$. 

Fix $\psi\in Z$. For each $\ve>0$,  there is $M\in\N$ with $|\psi(x)|\leq\frac{\ve}{3}$ whenever $x\geq M$ (as $\psi(\infty)=0$). By Weierstrass' Theorem there is a polynomial $Q(x)=\sum_{j=1}^ra_jx^j$ with  $Q(0)=0$ such that $|\psi(x)-Q(x)|\leq\frac{\ve}{3}$ for $x\in [0,M]$. Observe that $|Q(M)|\leq |Q(M)-\psi(M)|+|\psi(M)|\leq\frac{2\ve}{3}$. Moreover, the function $h:=\sum_{j=1}^ra_jg_{M,j}$ belongs to $(I-\sC)(C_l([0,\infty]))$ and coincides with $Q$ on $[0,M]$. Now, if $x\in [0,M]$, then
\[
|\psi(x)-h(x)|=|\psi(x)-Q(x)|\leq\frac{\ve}{3}
\]
and, if $x\geq M$, then
\begin{eqnarray*}
& & |\psi(x)-h(x)|=\left|\psi(x)-\sum_{j=1}^ra_j\frac{M^{j+1}}{x}\right|\leq |\psi(x)|+\frac{M}{x}\left|\sum_{j=1}^ra_jM^{j}\right|\\
& & =|\psi(x)|+\frac{M}{x}|Q(M)|\leq\frac{\ve}{3}+\frac{2\ve}{3}=\ve.
\end{eqnarray*}
Accordingly, $\|\psi-h\|_\infty\leq\ve$. It follows that  $\psi\in \ov{(I-\sC)(C_l([0,\infty]))}$. This completes the proof.
\end{proof}

\begin{remark}\label{r.T2-C}\rm  It is possible to use the identity \eqref{e.range} in Theorem \ref{T2-C} to give an alternate proof
 that $\sC$ is not mean ergodic in $C_l([0,\infty])$. Indeed, if $\sC$ is mean ergodic, then $C_l([0,\infty])=\Ker (I-\sC)\oplus \ov{(I-\sC)(C_l([0,\infty]))}$, \cite[Ch. 2]{K}, \cite[Ch. VIII, \S 3, Corollary]{Y}, and so the function $f(x)=(\cos x)/(x+1)\in C_l([0,\infty])$ could be written as $f=c\textbf{1}+g$ with $g(0)=g(\infty)=0$. This implies that  $f(0)=c=f(\infty)$. But, $f(0)=1$ and $f(\infty)=0$ which gives  a contradiction.
\end{remark}

We now identify the spectrum and  point spectrum of the Ces\`aro operator $\sC$ on the Banach spaces $C([0,1])$ and $C_l([0,\infty])$.

\begin{prop}\label{PS1-C} The Ces\`aro operator $\sC\colon C([0,1])\to C([0,1])$ satisfies
\[
\sigma(\sC; C([0,1]))=\left\{\lambda\in \C\colon \left|\lambda-\frac{1}{2}\right|\leq \frac{1}{2}\right\},
\]
and
\[
\sigma_{pt}(\sC; C([0,1]))=\left\{\lambda\in \C\colon \left|\lambda-\frac{1}{2}\right|\leq \frac{1}{2}\right\}\setminus\{0\}.
\]
\end{prop}

\begin{proof} It is routine to check that $\sC$ is injective on $C([0,1])$. Also, $\sC$  is not surjective (the range of $\sC$ contains only  continuously differentiable functions on $(0,1]$).  Hence, $0\in \sigma(\sC; C([0,1]))\setminus\sigma_{pt}(\sC; C([0,1]))$. If $\lambda\in \C\setminus\{0\}$ satisfies  $\left|\lambda-\frac{1}{2}\right|\leq \frac{1}{2}$, then the function $g_\lambda(x):=x^{\frac{1}{\lambda}-1}$, for $x\in [0,1]$, belongs to $C([0,1])$ and $\sC g_\lambda=\lambda g_\lambda$, i.e., $\lambda\in \sigma_{pt}(\sC; C([0,1]))$. If $\lambda\in \C$ satisfies  $\left|\lambda-\frac{1}{2}\right|> \frac{1}{2}$ (equivalently, ${\rm Re}\left(\frac{1}{\lambda}\right)<1$), then for $\xi:=\frac{1}{\lambda}$ the linear map
\[
P_\xi f(x):=\int_0^1 s^{-\xi}f(xs)\, ds,\quad x\in [0,1],
\]
is a bounded operator on $L^\infty([0,1])$ with the property that   $\xi I+\xi^2P_\xi$ is the inverse of $(\lambda I-\sC)$ on $L^\infty([0,1])$; see \cite{B} and the comments on p.29 of \cite{Le}. By the dominated convergence theorem applied to calculating $\lim_{n\to\infty}P_\xi f(x_n)$ whenever  $f\in C([0,1])$ and $x_n\to x$ in $[0,1]$ for $n\to\infty$, it follows that  $P_\xi f\in C([0,1])$ whenever $f\in C([0,1])$, i.e., $\xi I+\xi^2P_\xi$ restricted to the closed invariant subspace $C([0,1])$ of $L^\infty([0,1])$ is the inverse of  $(\lambda I-\sC)$ restricted from  $L^\infty([0,1])$ to $C([0,1])$. This implies that $\lambda\not\in \sigma(\sC;C([0,1]))$. So, the proof is complete. 
\end{proof}

\begin{prop}\label{PS2-C}The Ces\`aro operator $\sC\colon C_l([0,\infty])\to C_l([0,\infty])$ satisfies
\[
\sigma(\sC; C_l([0,\infty]))=\left\{\lambda\in \C\colon \left|\lambda-\frac{1}{2}\right|= \frac{1}{2}\right\}
\]
and
\[
\sigma_{pt}(\sC; C_l([0,\infty]))=\{1\}.
\]
\end{prop}

\begin{proof} Since $\sC\textbf{1}=\textbf{1}$, we have $1\in \sigma_{pt}(\sC; C_l([0,\infty]))$. The same argument given in the proof of Proposition \ref{PS1-C} yields that $0\in \sigma(\sC; C_l([0,\infty]))\setminus \sigma_{pt}(\sC; C_l([0,\infty]))$. If $f\in C_l([0,\infty])$ satisfies $\sC f=\lambda f$ for some $\lambda\not=0$, then $f$ is continuously differentiable in $(0,\infty)$ and is a solution of the 1-st order Euler differential equation $\lambda xy'(x)+(\lambda -1)y(x)=0$. But, every solution of this ODE has the form $f(x)=\beta x^{\frac{1}{\lambda}-1}$,  $x\in \R^+$, for some  $\beta\in\C$.  Since $x^\alpha\not\in C_l([0,\infty])$ unless $\alpha=0$,  we conclude that necessarily $\lambda=1$. Thus $\sigma_{pt}(\sC; C_l([0,\infty]))=\{1\}$.

Now we prove that   
\[
\sigma(\sC; C_l([0,\infty]))=\left\{\lambda\in \C\colon \left|\lambda-\frac{1}{2}\right|= \frac{1}{2}\right\}=\{\lambda\in\C\colon {\rm Re}\left(\frac{1}{\lambda}\right)=1\}.
\]
For this we recall that D. Boyd proved the following results, \cite[Theorem 1]{B}:
\begin{enumerate}
\item If ${\rm Re}(\xi)<1$, then $P_\xi f(x):=\int_0^1s^{-\xi}f(xs)\, ds$, for  $x\in \R^+$, defines a continuous linear operator on $L^\infty(\R^+)$. Moreover, %and (with $\lambda=\frac{1}{\xi}$), 
if $\lambda\not=0$ and  ${\rm Re}\left(\frac{1}{\lambda}\right)<1$, then
\[
(\lambda I-\sC)^{-1}=(\lambda^{-1}I+\lambda^{-2}P_{1/\lambda}), \ {\rm\ in\ } \cL(L^\infty(\R^+)).
\] 
\item If ${\rm Re}(\xi)>1$, then $Q_\xi f(x):=\int_1^\infty s^{-\xi} f(xs)\,ds$, for $x\in \R^+$, defines a continuous linear operator on $L^\infty(\R^+)$. Moreover,
% and (with $\lambda=\frac{1}{\xi}$), 
if $\lambda\not=0$ and ${\rm Re}\left(\frac{1}{\lambda}\right)>1$, then
\[
(\lambda I-\sC)^{-1}=(\lambda^{-1}I+\lambda^{-2}Q_{1/\lambda}), \ {\rm\ in\ } \cL(L^\infty(\R^+)).
\]
\end{enumerate}
To show that 
\begin{equation}\label{eq.prop-1}
\sigma(\sC; C_l([0,\infty]))\su \left\{\lambda\in\C\colon {\rm Re}\left(\frac{1}{\lambda}\right)=1\right\},
\end{equation}
we first observe, via the dominated convergence theorem (as applied in the proof of Proposition \ref{PS1-C}), that if $f$ is bounded and continuous on $\R^+$, then also $P_\xi f$, for ${\rm Re}(\xi)<1$, and $Q_\xi f$, for ${\rm Re}(\xi)>1$, are bounded and  continuous functions on $\R^+$. So, the proof of  \eqref{eq.prop-1} will follow if we can  show,  for each $f\in C_l([0,\infty])$, that $P_\xi f\in C_l([0,\infty])$ whenever ${\rm Re}(\xi)<1$ and that $Q_\xi f\in C_l([0,\infty])$ whenever ${\rm Re}(\xi)>1$. To see this, fix $f\in C_l([0,\infty])$. Take first $\xi\in\C$ with ${\rm Re}(\xi)<1$. Fix a sequence $\{x_n\}_{n=1}^\infty\su (0,\infty)$ such that $x_n\to \infty$. Then $f(x_n s)\to f(\infty)$,  for every  fixed $s\in (0,1)$ as $n\to\infty$, and $|s^{-\xi}f(x_ns)|\leq s^{-{\rm Re}(\xi)}\|f\|_\infty$ for all $s\in (0,1)$ and $n\in\N$ with $s^{- {\rm Re}(\xi)}\in L^1([0,1])$ as ${\rm Re}(\xi)<1$. Then the dominated convergence theorem implies that $P_\xi f(x_n)=\int_0^1 s^{-\xi} f(x_ns)\,ds\to \int_0^1 s^{-\xi}f(\infty)\,ds$ as $n\to\infty$. Since the sequence $\{x_n\}_{n=1}^\infty$ is arbitrary, it follows that $\lim_{x\to\infty}P_\xi f(x)$ exists and is equal to $\frac{f(\infty)}{1-\xi}$. So, $P_\xi f\in C_l([0,\infty])$.

Now, let  $\xi\in\C$ with ${\rm Re}(\xi)>1$. Fix again a sequence $\{x_n\}_{n=1}^\infty\su (0,\infty)$ such that $x_n\to \infty$. Then $f(x_n s)\to f(\infty)$,   for each fixed $s\in (1,\infty)$ as $n\to\infty$, and $|s^{-\xi}f(x_ns)|\leq s^{-{\rm Re}(\xi)}\|f\|_\infty$  for all  $s\in (1,\infty)$ and $n\in\N$ with $s^{-{\rm Re}(\xi)}\in L^1((1,\infty))$ as ${\rm Re}(\xi)>1$. Then the dominated convergence theorem implies that $Q_\xi f(x_n)=\int_1^\infty s^{-\xi} f(x_ns)\,ds\to \int_1^\infty s^{-\xi}f(\infty)\,ds$ as $n\to\infty$. Since the sequence $\{x_n\}_{n=1}^\infty$ is arbitrary, it follows that $\lim_{x\to\infty}Q_\xi f(x)$ exists and is equal to $\frac{f(\infty)}{\xi -1}$. So, $Q_\xi f\in C_l([0,\infty])$.

%To complete the proof we still have to show 
We will also require  the fact
that $\|P_\xi\|_{op}=\frac{1}{1-{\rm Re}(\xi)}$ for each $\xi\in\C$ with ${\rm Re}(\xi)<1$. To verify this, fix $\xi\in\C$ with ${\rm Re}(\xi)<1$. If $f\in C_l([0,\infty])$ and $x\in \R^+$, then it follows from the definition of $P_\xi f$ that  $|P_\xi f(x)|\leq \frac{1}{1-{\rm Re}(\xi)}\|f\|_\infty$. So, $\|P_\xi\|_{op}\leq \frac{1}{1-{\rm Re}(\xi)}$. Since $P_\xi \textbf{1}(x)=\frac{1}{1-\xi}$, for $x\in \R^+$, we are done if $\xi\in\R$ with $\xi<1$. For the remaining case, assume $\xi=\alpha+i\beta$, with $\beta\not=0$ and $\alpha<1$.  For given $\ve>0$, define $g_\ve(x):=x^{i\beta+\ve}$ if $x\in (0,1]$, $g_\ve(0):=0$ and $g_\ve(x):=1$ if $x\geq 1$. Then $g_\ve\in C_l([0,\infty])$ and $\|g_\ve\|_\infty=1$.  Moreover, 
\[
P_\xi g_\ve (1)=\int_0^1 s^{-\xi} g_\ve(s)\, ds=\int_0^1 s^{-\alpha+\ve}\, ds=\frac{1}{1-\alpha+\ve}.
\] 
So, $\|P_\xi g_\ve\|_\infty\geq \frac{1}{1-\alpha+\ve}$. Hence, $\|P_\xi \|_{op}\geq \sup_{\ve >0}\frac{1}{1-\alpha+\ve}=\frac{1}{1-\alpha}=\frac{1}{1-{\rm Re}(\xi)}$ and we can conclude, as stated, that $\|P_\xi \|_{op}= \frac{1}{1-{\rm Re}(\xi)}$.
 
We complete the proof of $\sigma(\sC; C_l([0,\infty]))=\{\lambda\in\C\colon {\rm Re}\left(\frac{1}{\lambda}\right)=1\}$ by applying an argument of  Boyd, \cite[p.34]{B}. Suppose there exists $\lambda_0\in\rho(\sC; C_l([0,\infty]))$ with $\lambda_0\in \C\setminus\{0,1\}$ and satisfying ${\rm Re}\left(\frac{1}{\lambda_0}\right)=1$. Select a sequence $\{\lambda_n\}_{n=1}^\infty\su \C$ such that $\lambda_n\to\lambda_0$ for $n\to\infty$ and ${\rm Re}\left(\frac{1}{\lambda_n}\right)<1$ for all $n\in\N$. For each $n\in\N$, set $\xi_n:=\frac{1}{\lambda_n}$. Then
\[
\|(\lambda_n I-\sC)^{-1}\|_{op}=\|\xi_n I+\xi_n^2 P_{\xi_n}\|_{op}\geq |\xi_n|^2\frac{1}{1-{\rm Re}(\xi_n)}-|\xi_n|
\]
for every $n\in\N$. Since ${\rm Re}(\xi_n)\to 1$ as $n\to\infty$, it follows that $\|(\lambda_n I-\sC)^{-1}\|_{op}\to \infty$ for $n\to\infty$. This is a contradiction because the resolvent set  $\rho(\sC; C_l([0,\infty]))$ is open in $\C$ and the resolvent map $\lambda\mapsto (\lambda I-\sC)^{-1}$ is   continuous from $\rho(\sC; C_l([0,\infty]))$ into $\cL_b(C_l([0,\infty]))$. So, no such $\lambda_0$ exists.
%Fix $\ve >0$ and select $M>0$ such that $|f(x)-f(\infty)|<\ve$ for all $x\geq M$. With the change of variable $xs=t$, we get, for $x\geq M$, that
%\begin{eqnarray*}
%& & \left|P_\xi f(x)-\frac{1}{1-\xi}f(\infty)\right|=\left|\frac{1}{x^{1-\xi}}\int_0^xt^{-\xi}f(t)\, dt-\frac{1}{x^{1-\xi}}\int_0^xt^{-\xi}f(\infty)\, dt\right|\\
%& & \leq \frac{1}{x^{1-{\rm Re}\xi}}\int_0
%\end{eqnarray*}
\end{proof}

\begin{prop}\label{P-C} The Ces\`aro operator $\sC\colon C_l([0,\infty])\to C_l([0,\infty])$ is not supercyclic.
\end{prop}

\begin{proof} The argument is similar to the one used in the proof of Proposition \ref{P.8}. One replaces the continuous function $f$ used there (i.e., $f(x)=x$, for $x\in [0,1]$) with $f\in C_l([0,\infty])$ given by $f(x)=x$ if $x\in [0,1]$ and $f(x)=1$ if $x\geq 1$, say.
\end{proof}

\section{The Ces\`aro operator on the Banach spaces $L^p([0,1])$ and $L^p(\R^+)$}

We now study  the Ces\`aro operator $\sC$, as given by \eqref{eq.I-1},
on the Banach spaces $L^p([0,1])$ and $L^p(\R^+)$, $1<p<\infty$. 
Hardy's inequality, \cite[p.240]{HLP}, ensures that both of the linear maps
$\sC\colon L^p([0,1])\to L^p([0,1])$ and $\sC\colon L^p(\R^+)\to L^p(\R^+)$, $1<p<\infty$, are continuous with operator norm $\|\sC\|_{op}=q$, where  $\frac{1}{p}+\frac{1}{q}=1$. 

In the sequel, the \textit{spectral radius} of an operator $T$ acting on a Banach space $X$ is denoted by $r(T):=\sup\{|\lambda|\colon \lambda\in \sigma(T)\}$. Recall that always $r(T)\leq \|T\|_{op}$, \cite[Ch. VIII, Theorems 2.3 and 2.4]{Y}.

\begin{theorem}\label{T1-L} The Ces\`aro operator $\sC\colon L^p([0,1])\to L^p([0,1])$, $1<p<\infty$, is not power bounded and not mean ergodic. On the other hand, it is hypercyclic, chaotic and satisfies
\[
\sigma(\sC;L^p([0,1]))=\left\{\lambda\in \C\colon \left|\lambda-\frac{q}{2}\right|\leq \frac{q}{2}\right\}
\]
and
\[
\sigma_{pt}(\sC;L^p([0,1]))=\left\{\lambda\in \C\colon \left|\lambda-\frac{q}{2}\right|< \frac{q}{2}\right\}.
\]
\end{theorem}

\begin{proof} The statements about the spectrum are due to G.M. Leibowitz, \cite{Le}; see also \cite[Theorem 1]{Le-2}.

Le\'on-Saavedra et al. have shown in \cite[Theorems 2.3 and 2.6]{LPS}, that $\sC$ is both hypercyclic and chaotic on $L^p([0,1])$. 

Fix $1<p<\infty$. Suppose that $\sC$ is mean ergodic on $L^p([0,1])$.  It follows from the identities $\frac{\sC^n}{n}=\sC_{[n]}-\frac{(n-1)}{n}\sC_{[n-1]}$, for $n\geq 2$,  that  $\frac{\sC^n}{n}\to 0$ in $\cL_s(L^p([0,1]))$ for $n\to\infty$. So, $\left\{\frac{\sC^n}{n}\right\}_{n\in\N}$ is uniformly bounded relative to $\|\cdot\|_{op}$
 (by the Principle of Uniform Boundedness). Hence, there exists $M>0$ such that $\left\|\frac{\sC^n}{n}\right\|_{op}\leq M$ for all $n\in\N$. On the other hand, 
by the spectral mapping theorem $\sigma(\sC^n; L^p([0,1]))=\{\lambda^n\colon \lambda\in \sigma(\sC;L^p([0,1]))\}$, for $n\in\N$, and so $q^n\in \sigma(\sC^n;L^p([0,1]))$, for $n\in\N$. Therefore,
\[
q^n\leq r(\sC^n)\leq \|\sC^n\|_{op}\leq Mn,\quad n\in\N;
\]
 a contradiction as $q>1$. Hence, $\sC$ is not mean ergodic.

From $q^n\in \sigma(\sC^n;L^p([0,1]))$ and $\|\sC\|_{op}=q$, for $n\in\N$, it follows that $\|\sC^n\|_{op}=q^n$ for each $n\in\N$. Hence, $\sC$ cannot be power bounded. 
\end{proof}

\begin{theorem}\label{T2-L} The Ces\`aro operator $\sC\colon L^p(\R^+)\to L^p(\R^+)$, $1<p<\infty$, is not power bounded and not mean ergodic. Moreover, 
\[
\sigma(\sC;L^p(\R^+))=\left\{\lambda\in \C\colon \left|\lambda-\frac{q}{2}\right|=\frac{q}{2}\right\}
\]
and
\[
\sigma_{pt}(\sC;L^p(\R^+)=\emptyset.
\]
\end{theorem}

\begin{proof} The statement about the spectrum is due to D. Boyd, \cite[Theorem 2]{B}; for the point spectrum we refer to \cite[p.28]{Le}.

Fix $1<p<\infty$. Suppose that $\sC$ is mean ergodic on $L^p(\R^+)$. Since $q\in \sigma(\sC;L^p(\R^+))$, we can proceed as  in the proof of Theorem \ref{T1-L} to deduce that $q^n\leq r(\sC^n)\leq \|\sC^n\|_{op}\leq Mn$ for each $n\in\N$ and some constant $M>0$; a contradiction as $q>1$. Arguing as in the proof of Theorem \ref{T1-L} we again deduce that $\|\sC^n\|_{op}=q^n$, for $n\in\N$. So, $\sC$ cannot be power bounded.  
\end{proof}

\begin{remark}\label{r.(infty)}\rm According to Corollary 3.3 of \cite{G-LS} the Ces\`aro operator $\sC$
 is not supercyclic in $L^2(\R^+)$. Moreover, $\sC$ does not have any non-zero periodic points in any of the spaces $L^p(\R^+)$, $1<p<\infty$. For, if so,
 then some  root of  unity would be an eigenvalue of $\sC$,  contrary to the fact that $\sigma_{pt}(\sC;L^p([0,1]))=\emptyset$. 
\end{remark}

%%%%%%%%%%%%%%%%%%%%%%%%%%%%%%%%%%%%%%%%%%%%%%%%%%%%%%%%%%%%%%%%%%e%%%%%%%%%%%%%%%%%%%%%%%%%%%%%%%%%%%%%%%%%%%%%%%%%%%%%%%%%%%%%%%%%%%%%%%%%%%%%%%%%%%%%%%%%%%%%%%%%%%%%%%%%%%%%%%%%%%%%%%%%%%%%%%%%%%%

\section{The Ces\`aro operator on the Fr\'echet space $C(\R^+)$}

%We denote by  $C[0,\infty)$  the space of all complex--valued continuous functions on $[0,\infty)$ endowed with the topology of uniform convergence on compact sets. 
The lc--topology of  the Fr\'echet space  $C(\R^+)$ (see \S 1) is  generated by the increasing sequence of seminorms
\begin{equation}\label{e.sem-c}
q_j(f):=\max_{x\in [0,j]}|f(x)|, \quad f\in C(\R^+),\ j\in\N. 
\end{equation}
For each $j\in\N$, we denote by $C([0,j])$ the Banach space of all $\C$-valued, continuous functions on $[0,j]$ endowed with the norm
\begin{equation}\label{e.sem-bb}
\|f\|_j:=\max_{x\in [0,j]}|f(x)|, \quad f\in C([0,j]).
\end{equation}

For each $j\in\N$,  let  $Q_j\colon C(\R^+)\to C([0,j])$  and $Q_{j,j+1}\colon C([0,j+1])\to C([0,j])$   be  the respective restriction maps, i.e., $Q_jf:=f|_{[0,j]}$ for $f\in C(\R^+)$ and $Q_{j,j+1}f:=f|_{[0,j]}$ for $f\in C([0,j+1])$. Clearly, $Q_{j,j+1}\circ Q_{j+1}=Q_j$ with  $\|Q_jf\|_{j}=q_j(f)$ and $\|Q_{j,j+1}g\|_{j}=\|g\|_{j}\leq \|g\|_{j+1}$ for every  $f\in C(\R^+)$,  $g\in C([0,j+1])$ and $ j\in\N$. Moreover, we have the projective limit $C(\R^+)=\proj_{j\in\N}(C([0,j]), Q_{j,j+1})$. Observe that all of the operators $Q_{j,j+1}$  and $Q_j$, for  $j\in\N$, are \textit{surjective}. 
%Then 
%\begin{eqnarray}
%\label{e.cont-c} \|R_jf\|_{\infty,j}=r_j(f), & &\quad f\in C[0,\infty),\  j\in\N,\\
%\label{e.cont-1c} \|R_{j.j+1}f\|_{\infty,j}=\|f\|_{\infty,j}\leq \|f\|_{\infty,j+1}, & & \quad f\in C[0,j+1], \  j\in\N,\\
%\label{e.comp-c} R_{j,j+1}R_{j+1}=R_j, & & \quad j\in\N.
%\end{eqnarray}
%So, $R_j\in\cL(C[0,\infty), C[0,j])$, $R_{j.j+1}\in \cL(C[0,j+1], C[0,j])$, for $j\in\N$, and $C[0,\infty)=\proj_{j\in\N}(C[0,j], R_{j,j+1})$. Since the operators $R_{j,j+1}$ are surjective (and so also the operators $R_j$ are surjective via \eqref{e.comp-c}), $C[0,\infty)$ is then a quojection Fr\'echet space. Actually, each operator   $R_{j,j+1}$ admits a continuous right-inverse $J_{j+1,j}\colon C[0,j]\to C[0,j+1]$ given by $J_{j+1,j}g:=\tilde{g}$ with
%\begin{equation}\label{e.est}
%\tilde{g}(x)=\left\{\begin{array}{ll}
%g(x) & \mbox{for $x\in [0,j]$}\\
%g(j-x)\varphi(x-j)& \mbox{for $x\in [j,j+1]$, }
%\end{array}
%\right.
%\end{equation}
%where $\varphi$ is a continuous function on $\R$ such that $0\leq \varphi\leq 1$ and ${\rm supp}\varphi\su [0,1]$.
%Analogously, $R_j$ admits a continuous right-inverse $J_{j}\colon C[0,j]\to C[0,\infty)$ given by $J_{j}g:=\tilde{g}$ where $\tilde{g}$ is defined according to \eqref{e.est}; observe that ${\rm supp}\tilde{g}\su [0,j+1]$. 

We investigate the Ces\`aro operator $\sC\colon C(\R^+)\to C(\R^+)$ defined, for every $f\in C(\R^+)$, by $\sC f(0)=f(0)$ and $\sC f(x)=\frac{1}{x}\int_0^xf(t)\,dt$, for $x>0$. To do this we denote by $\sC_j\colon C([0,j])\to C([0,j])$ the Banach space operator defined by the same formulae but, now for  $f\in C([0,j])$, $j\in\N$. 
%Next, let $C_\infty\colon C[0,\infty)\to C[0,\infty)$ denote the Ces\`aro operator defined according to \eqref{e.Cesaro} for $x\in (0,\infty)$ and by $Cf(0)=0$ for $x=0$.
It is routine to check that $\sC_jQ_j=Q_j\sC$ and $Q_{j,j+1}\sC_{j+1}=\sC_jQ_{j,j+1}$ for every $j\in\N$. The continuity of  $\sC$ on  $C(\R^+)$ is immediate from the inequalities $q_j(\sC f)\leq q_j(f)$, for $f\in C(\R^+)$ and $j\in\N$. 
For each $j\in\N$, define $T_j\colon C([0,1])\to C([0,j])$ by $T_jg(x):=g(x/j)$, for $x\in [0,j]$ and $g\in C([0,1])$. The linear operator $T_j$ is an isometry with inverse given by $T_j^{-1}h(x):=h(jx)$, for $x\in [0,1]$ and $h\in C([0,j])$. Moreover, $T_j\sC_1=\sC_jT_j$, for each $j\in\N$. To see this, fix $f\in C([0,1])$ and $x\in [0,j]$.  Since $(T_j\sC_1f)(0)=(\sC_jT_jf)(0)=f(0)$ we may assume that $x\in (0,j]$. Then 
\begin{eqnarray*}
& & (\sC_jT_jf)(x)=\frac{1}{x}\int_0^x(T_jf)(t)\,dt=\frac{1}{x}\int_0^x f(t/j)\, dt\\
& & =\frac{1}{x/j}\int_0^{x/j} f(s)\, ds=(\sC_1f)(x/j)=(T_j\sC_1)f(x).
\end{eqnarray*}
It follows from $T_j\sC_1=\sC_jT_j$ that $T_j\sC^n_1=\sC^n_jT_j$ and hence, that $T_j\sC^n_1T_j^{-1}=\sC^n_j$ for all $j,\ n\in\N$. Since both  $T_j$, $T_j^{-1}$ are isometries, we can conclude  
 that $\|\sC_j^n\|_{op}=\|\sC_1^n\|_{op}=1$, for each $j,\ n\in\N$, and that both
\begin{equation}\label{eq.par-1}
\sigma(\sC_j;C([0,j]))=\sigma(\sC_1; C([0,1]))
\end{equation}
and 
\begin{equation}\label{eq.par-2}
\sigma_{pt}(\sC_j;C([0,j]))=\sigma_{pt}(\sC_1; C([0,1])),
\end{equation}
for each  $j\in\N$.
 So, for each $j\in\N$,  the operator $\sC_j$ is power bounded. Moreover, the identities $(\sC_j)_{[n]}=T_j(\sC_1)_{[n]}T^{-1}_j$, for $j,\ n\in\N$, together with Theorem \ref{T1-C} and Proposition \ref{P.8} imply,  for each $j\in\N$, that $\sC_j$ is  mean ergodic but, not uniformly mean ergodic and not supercyclic (hence, not hypercyclic).
%\[
%(TC_1f)(x)=(C_1f)(x/a)=\frac{a}{x}\int_0^{x/a} f(s)\, ds, \quad x\in (0,a).
%\]

\begin{theorem}\label{T.7} The Ces\`aro operator $\sC\colon C(\R^+)\to C(\R^+)$ 
%on the Fr\'echet space $C[0,\infty)$
 is power bounded and mean ergodic but, not uniformly mean ergodic and  not supercyclic (hence, not hypercyclic). Moreover,
\[
\sigma(\sC; C(\R^+))=\left\{\lambda\in \C\colon \left|\lambda-\frac{1}{2}\right|\leq \frac{1}{2}\right\}
\]
and
\[
\sigma_{pt}(\sC; C(\R^+))=\sigma(\sC; C(\R^+))\setminus\{0\}.
\]
\end{theorem}

\begin{proof} All the assumptions of Lemmas \ref{l.spectra} and \ref{l.unif-mean} (in the Appendix) are satisfied with $X:=C(\R^+)$, $X_j:=C([0,j])$, $S:=\sC\in \cL(X)$ and $S_j:=\sC_j\in \cL(X_j)$, for $j\in\N$.

By Theorem \ref{T1-C} the operator $\sC_1\colon C([0,1])\to C([0,1])$ is power bounded and mean ergodic. The  comments prior to Theorem \ref{T.7} ensure  that $\sC_j\colon C([0,j])\to C([0,j])$ is also power bounded and mean ergodic, for each $j\in\N$. So, Lemma \ref{l.unif-mean} (i)\&(iii) in the Appendix yield that $\sC\colon C(\R^+)\to C(\R^+)$ is both power bounded and mean ergodic. 

If $\sC$ were uniformly mean ergodic on $C(\R^+)$, then also $\sC_1\colon C([0,1])\to C([0,1])$ would be uniformly mean ergodic by Lemma \ref{l.unif-mean}(ii) in the Appendix. This contradicts Theorem \ref{T1-C}. So, $\sC$ is not uniformly mean ergodic. 

Observe that $\sC_1Q_1=Q_1\sC$ with $Q_1$  surjective. If $\sC\colon C(\R^+)\to C(\R^+)$ is supercyclic, then $\{\lambda \sC^nf\colon n\in\N_0,\ \lambda\in \C\}$ is dense in  $C(\R^+)$ for some $f\in C(\R^+)$. By the properties mentioned in the previous two sentences it follows, with $g:=Q_1f\in C([0,1])$, that  $\{\lambda \sC_1^ng\colon n\in\N_0,\ \lambda\in \C\}$ is dense in  $C([0,1])$, i.e., 
$\sC_1$ is  supercyclic in $C([0,1])$. This contradicts Proposition \ref{P.8}. So, $\sC$ is not supercyclic in $C(\R^+)$.

Concerning the spectra,  by \eqref{e.resol-l} of Lemma \ref{l.spectra} (in the Appendix) and Proposition \ref{PS1-C} we have via \eqref{eq.par-1} that 
\begin{equation}\label{eq.par-3}
\sigma(\sC; C(\R^+))\su  \cup_{j=1}^\infty \sigma(\sC_j; C([0,j]))=\sigma(\sC_1; C([0,1])).
\end{equation}
%& = &\left\{\lambda\in \C\colon \left|\lambda-\frac{1}{2}\right|\leq \frac{1}{2}\right\}.
On the other hand, for every $\lambda\in \sigma_{pt}(\sC_1;C([0,1]))=\left\{\lambda\in \C\colon \left|\lambda-\frac{1}{2}\right|\leq\frac{1}{2}\right\}\setminus\{0\}$ (see Proposition \ref{PS1-C}), the function $f_\lambda(x)=x^{\frac{1}{\lambda}-1}$, for $x\in [0,1]$, when defined by the same formula for all $x\in \R^+$,  belongs to $C(\R^+)$ and satisfies $\sC f_\lambda=\lambda f_\lambda$. So, $\lambda\in \sigma_{pt}(\sC;C(\R^+))$ and  we have
\begin{equation}\label{eq.par-4}
\sigma_{pt}(\sC_1;C([0,1]))\su \sigma_{pt}(\sC;C(\R^+))\su \sigma(\sC;C(\R^+)).
\end{equation}
Since the range $\sC(C(\R^+))\su C^1(\R^+)$ with  $C^1(\R^+)$  a proper subspace of $C(\R^+)$, we see that $\sC$ is not surjective and so  $0\in \sigma(\sC;C(\R^+))$. So, we also have via Proposition \ref{PS1-C} that
\[
\sigma_{pt}(\sC_1;C([0,1]))\cup\{0\}=\sigma(\sC_1;C([0,1]))\su \sigma(\sC;C(\R^+)).
\]
By  \eqref{e.resol-l} of Lemma \ref{l.spectra} (in the Appendix) and \eqref{eq.par-2}, \eqref{eq.par-4}, it follows that
\begin{eqnarray*}
& & \sigma(\sC_1;C([0,1]))\su \sigma(\sC;C(\R^+))\cup\cup_{j=1}^\infty\sigma_{pt}(\sC_j; C([0,j]))\\
& & =\sigma(\sC;C(\R^+))\cup\sigma_{pt}(\sC_1; C([0,1]))\su \sigma(\sC;C(\R^+)).
\end{eqnarray*}
Combined with \eqref{eq.par-3} this yields that 
\[
\sigma(\sC;C(\R^+))=\sigma(\sC_1;C([0,1]))=\left\{\lambda\in \C\colon \left|\lambda-\frac{1}{2}\right|\leq\frac{1}{2}\right\}.
\]

Finally, by  \eqref{e.resol-ln} of Lemma \ref{l.spectra} (in the Appendix), we have
\[
\sigma_{pt}(\sC;C(\R^+))\su \cup_{j=1}^\infty\sigma_{pt}(\sC_j; C([0,j]))=\sigma_{pt}(\sC_1;C([0,1]))\su \sigma_{pt}(\sC;C(\R^+)).
\]
Thus, from  Proposition \ref{PS1-C} it follows that 
\[
\sigma_{pt}(\sC;C(\R^+))=\sigma_{pt}(\sC_1;C([0,1]))=\left\{\lambda\in \C\colon \left|\lambda-\frac{1}{2}\right|\leq\frac{1}{2}\right\}\setminus\{0\}.
\]
\end{proof}

%Ces\`aro operatorThe behaviour of the Ces\`aro operator $C$ on the Fr\'echet space $C[0,\infty)$ should be compared with the one of $C$ on the Banach psace $$
%%%%%%%%%%%%%%%%%%%%%%%%%%%%%%%%%%%%%%%%%%%%%%%%%%%%%%%%%%%%%%%%%%%%%%%%%%%%%%%%%%%%%%%%%%%%%%%%%%%%%%%%%%%%%%%%%%%%%%%%%%%%%%%%%%%%%%%%%%%%%%%%%%%%%%%%%%%%%%%%%%%%%%%%%%%%%%%%%%%%%%%%%%%%%%%%%%%%%%
\section{The Ces\`aro operator on the Fr\'echet space $L^p_{loc}(\R^+)$, $1<p<\infty$}

Recall that $L^p_{loc}(\R^+)$,   $1<p<\infty$, is  the Fr\'echet space of all  $\C$-valued,  measurable functions $f$  on $\R^+$ such that 
\begin{equation}\label{e.sem}
q_j(f):=\left(\int_0^j|f(x)|^p\, dx\right)^{1/p}<\infty, \quad j\in\N,
\end{equation}
 endowed with   the lc-topology generated by the increasing sequence of seminorms $\{q_j\}_{j\in\N}$ . 

Fix $1<p<\infty$. For each $j\in\N$,  denote by $L^p([0,j])$  the  Banach space of all $\C$-valued,  measurable functions   on $[0,j]$ with the norm $\|f\|_j:=\left(\int_0^j|f(x)|^p\, dx\right)^{1/p}$, for $f\in L^p([0,j])$.

For each $j\in\N$, denote by   $Q_j\colon L^p_{loc}(\R^+)\to L^p([0,j])$ and  $Q_{j,j+1}\colon L^p([0,j+1])\to L^p([0,j])$  the respective restriction maps on $[0,j]$, i.e., $Q_jf:=f|_{[0,j]}$ for $f\in L^p_{loc}(\R^+)$ and $Q_{j,j+1}f:=f|_{[0,j]}$ for $f\in L^p([0,j+1])$. Clearly, for each $j\in\N$, we have  $Q_{j,j+1}\circ Q_{j+1}=Q_j$ with $\|Q_jf\|_{j}=q_j(f)$,  for $f\in L^p_{loc}(\R^+)$, and  $\|Q_{j,j+1}g\|_{j}=\|g\|_{j}\leq \|g\|_{j+1}$,  for $g\in L^p([0,j+1])$. Observe that the maps $Q_j$ and $Q_{j,j+1}$ are \textit{surjective} for all $j\in\N$. Moreover,  $L_{loc}^p(\R^+)=\proj_{j\in\N}(L^p([0,j]), Q_{j,j+1})$. 
%\begin{eqnarray}
%\label{e.cont} \|R_jf\|_{p,j}=r_j(f), & &\quad f\in L^p_{loc}(0,\infty),\  j\in\N,\\
%\label{e.cont-1} \|R_{j.j+1}f\|_{p,j}=\|f\|_{p,j}\leq \|f\|_{p,j+1}, & & \quad f\in L^p(0,j+1), \  j\in\N,\\
%\label{e.comp} R_{j,j+1}R_{j+1}=R_j, & & \quad j\in\N.
%\end{eqnarray}
%So, $R_j\in\cL(L^p_{loc}(0,\infty), L^p(0,j))$, $R_{j.j+1}\in \cL(L^p(0,j+1), L^p(0,j))$, for $j\in\N$, and $L_{loc}^p(0,\infty)=\proj_{j\in\N}(L^p(0,j), R_{j,j+1})$. Since the operators $R_{j,j+1}$ are surjective (and so also the operators $R_j$ are surjective via \eqref{e.comp}), $L^p_{loc}(0,\infty)$ is then a quojection Fr\'echet space. Actually, each operator   $R_{j,j+1}$ admits a continuous right-inverse $J_{j+1,j}\colon L^p(0,j)\to L^p(0,j+1)$ given by $(J_{j+1,j}g)(x):=g(x)$ for $x\in (0,j)$ and $(J_{j+1,j}g)(x):=0$ otherwise (analogously, $R_j$ admits a continuous right-inverse $J_{j}\colon L^p(0,j)\to L^p_{loc}(0,\infty)$ given by $(J_{j}g)(x):=g(x)$ for $x\in (0,j)$ and $(J_{j}g)(x):=0$ otherwise). We observe that $J_{j+1, j}$ and $J_j$ are isomorphism into. So, $J_j(L^p(0,j))\simeq L^p(0,j)$ is a Banach closed subspace of $L^p_{loc}(0,\infty)$ whose lc-topology is generated by $r_j$. 

We consider the Ces\`aro operator $\sC\colon L_{loc}^p(\R^+)\to L_{loc}^p(\R^+)$  given by $\sC f(x):=\frac{1}{x}\int_0^x f(t)\,dt$, for $x>0$ and all  $f\in L^p_{loc}(\R^+)$, which is well defined as $L^p([0,x])\su L^1([0,x])$ for each $x>0$. For each $j\in\N$, denote by $\sC_j$ the operator defined in the same way on the Banach space $L^p([0,j])$. By Hardy's inequality, \cite[p.240]{HLP}, the linear  operators $\sC$ and $\sC_j$, $j\in\N$, are continuous. Moreover, it is routine to check that $\sC_jQ_j=Q_j\sC$ and $Q_{j,j+1}\sC_{j+1}=\sC_jQ_j$, for each $j\in\N$. More detailed information about the Fr\'echet space $L^p_{loc}(\R^+)$ can be found in \cite{A-1}, \cite{A-3}, \cite{A-2}, for example. 

Fix $j\in\N$ and define $T_j\colon L^p([0,1])\to L^p([0,j])$ by $(T_jf)(x):=f(x/j)$, for $x\in [0,j]$ and  $f\in L^p([0,1])$. Then the linear operator $T_j$ is a bijection with norm  $\|T_j\|_{op}=j^{1/p}$. Indeed, for every $f\in L^p([0,1])$, we have
\[
\|Tf\|_{j}^p=\int_0^j|(Tf)(x)|^p\, dx=\int_0^j |f(x/j)|^p\, dx=j\int_0^1 |f(y)|^p\, dy=j\|f\|^p_{1}.
\]
The inverse of $T_j$ is the operator $T_j^{-1}\colon  L^p([0,j])\to  L^p([0,1])$ given by  $(T_j^{-1}f)(x):=f(jx)$, for $x\in [0,1]$ and $f\in L^p([0,j])$, with   $\|T_j^{-1}\|_{op}=j^{-1/p}$.

The same calculations as  in \S 4 show  that 
$\sC_jT_j=T_j\sC_1$, for $j\in\N$. It follows that $T_j\sC_1^n=\sC_j^nT_j$ and hence, also that $T_j\sC_1^nT_j^{-1}=\sC_j^n$ for all $j,\ n\in\N$. Accordingly, $\|\sC_j^n\|_{op}\leq \|T_j\|_{op}\|\sC_1^n\|_{op}\|T_j^{-1}\|_{op}=\|\sC_1^n\|_{op}$. In a similar way it follows from $\sC_1^n=T_j^{-1}\sC_j^nT_j$ that $\|\sC_1^n\|_{op}\leq\|\sC_j^n\|_{op}$ and hence,
 for every $j,\ n\in\N$,  that $\|\sC_j^n\|_{op}=\|\sC_1^n\|_{op}=q^n$, where $\frac{1}{p}+\frac{1}{q}=1$.  It also follows that 
\begin{equation}\label{eq.par5-1}
\sigma(\sC_j; L^p([0,j]))=\sigma(\sC_1; L^p([0,1]))
\end{equation}
 and  
\begin{equation}\label{eq.par5-2}
\sigma_{pt}(\sC_j; L^p([0,j]))=\sigma_{pt}(\sC_1; L^p([0,1])),
\end{equation}
 for each $j\in\N$. Via the identities $(\sC_j)_{[n]}=T_j(\sC_1)_{[n]}T^{-1}_j$ and $\|\sC_j^n\|_{op}=\|\sC_1^n\|_{op}$, valid for all $j,\ n\in\N$, it follows from 
 Theorem \ref{T1-L} that  the operator $\sC_j\colon  L^p([0,j])\to L^p([0,j])$ is not power bounded and  not mean ergodic but,  it is hypercyclic and chaotic, for each $j\in\N$. 

\begin{theorem}\label{T1-Lp}Let $1<p<\infty$.  The Ces\`aro operator $\sC\colon L^p_{loc}(\R^+)\to  L^p_{loc}(\R^+)$  is not power bounded and not mean ergodic but, it is hypercyclic, chaotic and satisfies
\[
\sigma(\sC; L^p_{loc}(\R^+))=\left\{\lambda\in\C\colon \left|\lambda-\frac{q}{2}\right|\leq \frac{q}{2}\right\}
\]
and
\[
\sigma_{pt}(\sC; L^p_{loc}(\R^+))=\left\{\lambda\in\C\colon \left|\lambda-\frac{q}{2}\right|< \frac{q}{2}\right\}.
\]
\end{theorem}

\begin{proof} All the assumptions of Lemmas \ref{l.spectra} and \ref{l.unif-mean} (in the Appendix) are satisfied with  $X:=L^p_{loc}(\R^+)$, $X_j:=L^p([0,j])$,  $S:=\sC\in \cL(X)$ and $S_j:=\sC_j\in \cL(X_j)$, for $j\in\N$.

By Theorem \ref{T1-L} the operator $\sC_1\colon L^p([0,1])\to L^p([0,1])$ is neither power bounded nor mean ergodic. So, by applying Lemma \ref{l.spectra} (i)\&(iii) (in the Appendix) we can conclude that also $\sC\colon L^p_{loc}(\R^+)\to  L^p_{loc}(\R^+)$ is not power bounded and not mean ergodic.

Recalling that  $\sC_j$ is hypercyclic on $L^p([0,j])$ for all $j\in\N$ (cf. Theorem \ref{T1-L} and the  comments prior to Theorem \ref{T1-Lp}) and that $\sC_jQ_j=Q_j\sC$, for each $j\in\N$, we can apply \cite[Proposition 2.1]{BFPW} to conclude that $\sC$ is hypercyclic on $L^p_{loc}(\R^+)$. 

By \eqref{e.resol-l} of Lemma \ref{l.spectra} and Theorem \ref{T1-L} we have  via \eqref{eq.par5-1} that 
\begin{equation}\label{eq.par5-3}
 \sigma(\sC; L^p_{loc}(\R^+))\su \cup_{j=1}^\infty \sigma(\sC_j; L^p([0,j]))=\sigma(\sC_1; L^p([0,1])).
\end{equation}
%& & =\left\{\lambda\in\C\colon \left|\lambda-\frac{q}{2}\right|\leq \frac{q}{2}\right\}.
Now, if $\lambda\in\sigma_{pt}(\sC_1;L^p([0,1]))=\left\{\lambda\in\C\colon  \left|\lambda-\frac{q}{2}\right|<\frac{q}{2}\right\}$ (see Theorem \ref{T1-L}), then ${\rm Re}\left(\frac{1}{\lambda}\right)>\frac{1}{q}$ and so the function $f_\lambda(x):=x^{\frac{1}{\lambda}-1}$ belongs to $L^p_{loc}(\R^+)$ and is an eigenvector of $\sC$ corresponding to the eigenvalue $\lambda$. To see this, observe for every $j\in\N$ that
\[
(q_j(f_\lambda))^p=\int_0^j|x^{\frac{1}{\lambda}-1}|^p\, dx=\int_0^jx^{p\left({\rm Re}(\frac{1}{\lambda})-1\right)}\, dx <\infty,
\]
as $p\left({\rm Re}\left(\frac{1}{\lambda}\right)-1\right)>p(\frac{1}{q}-1)=-1$. Thus, $f_\lambda\in L^p_{loc}(\R^+)$. It is routine to check that $\sC f_\lambda=\lambda f_\lambda$. Hence,
\[
\sigma_{pt}(\sC_1;L^p([0,1]))\su \sigma_{pt}(\sC; L^p_{loc}(\R^+))\su \sigma(\sC; L^p_{loc}(\R^+)).
\]
So, by \eqref{e.resol-l} of Lemma \ref{l.spectra} it follows that 
\begin{eqnarray*}
& & \sigma(\sC_1;L^p([0,1]))\su \sigma(\sC; L^p_{loc}(\R^+))\cup \cup_{j=1}^\infty \sigma_{pt}(\sC_j;L^p([0,j]))\\
& & =\sigma(\sC;L^p_{loc}(\R^+))\cup\sigma_{pt}(\sC_1;L^p([0,1]))\su \sigma(\sC; L^p_{loc}(\R^+)).
\end{eqnarray*}
Combined with \eqref{eq.par5-3} this shows that 
\[
\sigma(\sC;L^p_{loc}(\R^+))= \sigma(\sC_1;L^p([0,1]))=\left\{\lambda\in\C\colon \left|\lambda-\frac{q}{2}\right|\leq\frac{q}{2}\right\}.
\]
Now,  by \eqref{e.resol-ln} of Lemma \ref{l.spectra}  we obtain
\[
\sigma_{pt}(\sC;L^p_{loc}(\R^+))\su \cup_{j=1}^\infty \sigma_{pt}(\sC_j;L^p([0,j]))=\sigma_{pt}(\sC_1;L^p([0,1]))\su \sigma_{pt}(\sC;L^p_{loc}(\R^+)),
\]
and hence, that 
\[
\sigma_{pt}(\sC;L^p_{loc}(\R^+))=\sigma_{pt}(\sC_1;L^p([0,1]))=\left\{\lambda\in\C\colon \left|\lambda-\frac{q}{2}\right|<\frac{q}{2}\right\}.
\]
\end{proof}

We already know that $\sC$ is hypercyclic. It remains to show that $\sC$ is chaotic in $L^p_{loc}(\R^+)$.  For this we need the following result.

\begin{lemma}\label{P.5} Let $1<p<\infty$ and the sequence $\{\alpha_n\}_{n=1}^\infty\su\C$ satisfy ${\rm Re}(\alpha_n)>-\frac{1}{p}$ for each $n\in\N$. Suppose that $\{\alpha_n\}_{n=1}^\infty$ has an accumulation point in the open set $H_p^+=\left\{z\in \C\colon {\rm Re}(z) >-\frac{1}{p}\right\}$. Then 
\[
Y:={\rm span}\{x^{\alpha_n}\colon n\in\N\}
\]
is a dense subspace of  $L^p_{loc}(\R^+)$.
\end{lemma}

\begin{proof} For each $\alpha\in H_p^+$, set $f_\alpha:=x^{\alpha}$, in which case ${\rm Re}(\alpha)>-\frac{1}{p}$ ensures that $f_\alpha\in L^p_{loc}(\R^+)$. Suppose that 
%the continuous linear functional 
$F\in (L^p_{loc}(\R^+))'$ satisfies $F(g)=0$ for each  $g\in Y$. Then 
$F(f_{\alpha_n})=0$, for each  $n\in\N$. By the structure of the dual space of $L^p_{loc}(\R^+)$,  \cite{A-3}, there exist $j\in\N$ and $h\in L^q(\R^+)$, where  $\frac{1}{p}+\frac{1}{q}=1$, such that ${\rm supp}(h)\su [0,j]$ and 
\[
F(f)=\int_0^j f(x)h(x)\, dx,\quad f\in L^p_{loc}(\R^+).
\]
Define 
$\Phi\colon H_p^+\to \C$ by  $\Phi(\alpha):=F(f_\alpha)=\int_0^j x^\alpha h(x)\, dx$, for  $\alpha\in H_p^+$. The function $\Phi$ is analytic on $H_p^+$ and vanishes on the sequence $\{\alpha_n\}_{n=1}^\infty$, which  has an accumulation point in $H_p^+$. So, $\Phi$ is identically  zero on $H_p^+$, i.e., $F(f_\alpha)=0$ for all $\alpha\in H_p^+$. In particular, $F$ vanishes on all $\C$-valued polynomials on $\R^+$.   Since such polynomials form a dense subspace of $L^p_{loc}(\R^+)$, it follows that $F=0$ on $L^p_{loc}(\R^+)$. As $F$ is arbitrary, we can conclude via the  Hahn--Banach theorem that $Y$ is dense in $L^p_{loc}(\R^+)$.
%
%To show the  denseness of the polynomials in $L^p_{loc}(0,\infty)$ we can proceed as it follows. Let $f\in  L^p_{loc}(0,\infty)$, $j\in\N$ and $\ve>0$ be fixed. Then there exists a function $g\in C[0,j]$ such that $r_j(f-g)=\|f-g\|_{p,j}<\ve/2$, being $C[0,j]$  a dense subspace of $L^p(0,j)$. By the classical Weierstrass approximation theorem there exists a polynomial $P$ such that $\sup_{x\in [0,j]}|g(x)-P(x)|<\ve/(2j)$. So $r_j(f-P)\leq r_j(f-g)+r_j(g-P)<\ve/2+\ve/2=\ve$. 
%
%By the arbitrariness of $F$ the proof follows via Hahn--Banach theorem.
\end{proof}

Returning to showing that $\sC$ is chaotic in $L^p_{loc}(\R^+)$, 
  it suffices  to verify that the space
\[
H:={\rm span}\{\Ker (\lambda I-\sC)\colon  \lambda=e^{2\pi i\theta} {\rm \ for \ some \ } \theta\in \Q\}
\]
is dense in $L^p_{loc}(\R^+)$; see \cite[Proposition 2.33]{GE-P}.
% and then to apply  the Godefroy--Shapiro criterion, \cite{GS}, \cite[Theorem 3.1]{GE-P}. 
We already know that 
\[
\sigma_{pt}(\sC;L^p_{loc}(\R^+))=\left\{\lambda\in\C\colon \left|\lambda-\frac{q}{2}\right|<\frac{q}{2}\right\}.
\]
Since $q>1$, we can select $\{\theta_n\}_{n=1}^\infty\su\Q$ such that  $\lambda_n:=e^{2\pi i\theta_n}\in  \sigma_{pt}(\sC;L^p_{loc}(\R^+))$ for each $n\in\N$ with $\lambda_n\to 1$ as $n\to\infty$. Define $\beta_n:=\frac{1}{\lambda_n}-1$, for $n\in\N$. Since $\lim_{n\to \infty}\beta_n=0$, also $\lim_{n\to\infty}{\rm Re}(\beta_n)=0$ and so there exists $N\in\N$ such that ${\rm Re}(\beta_n)>-\frac{1}{p}$ for all $n>N$. Hence, the sequence $\alpha_n:=\beta_{n+N}$, for $n\in\N$, satisfies ${\rm Re}(\alpha_n)>-\frac{1}{p}$ for all $n\in\N$ and $\{\alpha_n\}_{n=1}^\infty$ has $0\in H_p^+$ as an accumulation point. Then, 
by Lemma \ref{P.5} applied to $\{\alpha_n\}_{n=1}^\infty$,    the space
$
Y:={\rm span}\{x^{\alpha_n}\colon n\in\N\}$ 
is dense in $L^p_{loc}(\R^+)$. Since $\sC x^{\alpha_n}=\lambda_nx^{\alpha_n}$ with $x^{\alpha_n}\in L^p_{loc}(\R^+)$, for all $n\in\N$, it follows that  $Y\su H$. Hence, $\sC$ has a dense set of periodic points, i.e., it is chaotic (being also hypercyclic).\qed

%%%%%%%%%%%%%

%%%%%%%%%%%%%%%%%%%%%%%%%%%%%%%%%%%%%%%%%%%%%%%%%%%%%%%%%%%%%%%%%%%%%%%%%%%%%%%%%%%%%%%%%%%%%%%%%%%%%%%%%%%%%%%%%%%%%%%%%%%%%%%%%%%%%%%%%%%%%%%%%%%%%%%%%%%%%%%%%%%%%%%%%%%%%%%%%%%%%%%%%%%%%%%%%%%%%%

\section{Appendix}
Here we collect a few relevant results concerning the spectrum  and mean ergodic properties of continuous linear operators defined on %\textit{quojection}
 certain classes of Fr\'echet spaces.

\begin{lemma}\label{l.spectra} Let $X$ be a  Fr\'echet space and $S\in\cL(X)$. Suppose that   $X=\proj_{j\in\N}(X_j, Q_{j,j+1})$, with $X_j$ a Banach space (having norm $\|\ \|_j$) and linking maps $Q_{j,j+1}\in \cL(X_{j+1},X_j)$ which are surjective for all $j\in\N$, and suppose, for each $j\in\N$, that  there exists $S_j\in \cL(X_j)$ satisfying  
\begin{equation}\label{e.resol-1}
S_jQ_j=Q_jS,
\end{equation}
where $Q_j\in \cL(X,X_j)$, $j\in\N$, denotes the canonical projection of $X$ onto $X_j$ (i.e., $Q_{j,j+1}\circ Q_{j+1}=Q_j$).
Then 
\begin{equation}\label{e.resol-l}
\sigma(S)\su \cup_{j=1}^\infty \sigma(S_j)\su \sigma(S)\cup \cup_{j=1}^\infty \sigma_{pt}(S_j).
\end{equation}
Moreover, 
\begin{equation}\label{e.resol-ln}
\sigma_{pt}(S)\su \cup_{j=1}^\infty \sigma_{pt}(S_j).
\end{equation}
\end{lemma}

\begin{proof} It follows from  \eqref{e.resol-1} that
\begin{equation}\label{e.resol-2}
(\lambda I_j-S_j)Q_j=Q_j(\lambda I-S)
\end{equation}
for all $j\in\N$ and $\lambda\in \C$, where $I_j$ denotes the identity map on $X_j$.

Fix any  $\lambda\in \cap_{j=1}^\infty\rho(S_j)$. If $(\lambda I-S)x=0$ for some $x\in X$, then by \eqref{e.resol-2} we have $(\lambda I_j-S_j)Q_jx=Q_j(\lambda I-S)x=0$  for all $j\in\N$. It follows that $Q_jx=0$ for all $j\in\N$. This implies that $x=0$ as $x\in X=  \proj_{j\in\N}(X_j, Q_{j,j+1})$. The proof of  the surjectivity of $(\lambda I-S)$ follows as in the last part of the proof (cf. p.154) of Theorem 4.1 of \cite{ABR-4} via \eqref{e.resol-2} and the fact that $(\lambda I_j-S_j)$ is bijective for all $j\in\N$. As $X$ is a Fr\'echet space, we can conclude that $(\lambda I-S)\in \cL(X)$ and so $\lambda\in \rho(S)$. This establishes that $\sigma(S)\su \cup_{j=1}^\infty \sigma(S_j)$.

%Suppose, in addition, that \eqref{e.resol-ln} holds for each $\lambda\in \rho(S)$.
To verify the second containment in  \eqref{e.resol-l} we first observe that if  $\mu\in \rho(S)$, then $(\mu I- S)$  is invertible in $\cL(X)$ and hence, $(\mu I_j- S_j)\in \cL(X_j)$ is \textit{surjective} for all $j\in\N$; this follows routinely from \eqref{e.resol-2}  and the fact that each operator $Q_j$, for $j\in\N$, is surjective. Suppose that $\nu\in \rho(S)\setminus \cap_{j=1}^\infty \rho(S_j)$. Then $\nu\not\in\rho(S_{j_0})$ for some $j_0\in \N$, i.e., $(\nu I_{j_0}-S_{j_0})$ is \textit{not} invertible in $\cL(X_{j_0})$. Since $(\nu I_{j_0}-S_{j_0})$ is   surjective, it follows that $\nu\in \sigma_{pt}(S_{j_0})$.

Now, let $\lambda\in \cup_{j=1}^\infty \sigma(S_j)$. If $\lambda\in \sigma(S)$, then there is nothing to prove. If $\lambda\not\in \sigma(S)$, then $\lambda\in \rho(S)$. From the previous paragraph $\lambda\in \sigma_{pt}(S_{j_0})$ for some $j_0\in \N$, i.e., $\lambda\in \cup_{j=1}^\infty \sigma(S_j)$. This establishes the second containment in  \eqref{e.resol-l}. Thereby  \eqref{e.resol-l} has been proved.

To verify \eqref{e.resol-ln} let $\lambda\in (\cup_{j=1}^\infty \sigma_{pt}(S_j))^c$, in which case $(\lambda_jI_j-S_j)$
is injective for each $j\in\N$. Suppose that $x\in X$ satisfies $(\lambda I- S)x=0$ in which case \eqref{e.resol-2} implies that 
$(\lambda I_j-S_j)Q_jx=0$ for every $j\in\N$. Hence, $Q_jx=0$ for every $j\in\N$ and so $x=0$. This shows that   $(\lambda I-S)$ is injective and so $\lambda\not\in \sigma_{pt}(S)$, i.e.,
 $\lambda\in (\sigma_{pt}(S))^c$. Thereby  \eqref{e.resol-ln} is established.
\end{proof}

A Fr\'echet space $X$ is always a projective limit of  continuous linear operators $R_{j}:\ X_{j+1}\to X_j$, for $j\in\N$, with each $X_j$ a Banach space.  If $X_j$ and $R_j$ can be chosen such that each $R_j$ is surjective and $X$ is isomorphic to the projective limit $\proj_{j\in\N}(X_j,R_j)$, then $X$ is called a \textit{quojection}, \cite[Section 5]{BD}. 
 Banach spaces and countable products of Banach spaces are quojections.
% Actually, every quojection is the quotient of a countable product of Banach spaces, \cite{BMMMV}. 
In \cite{M} Moscatelli gave the first examples of quojections which are not isomorphic to countable products of Banach spaces. Concrete examples of quojection Fr\'echet spaces are  $\omega=\C^\N$, the  spaces $L^p_{loc}(\Omega)$, for $1\leq p\leq \infty$, and $C^{(m)}(\Omega)$ for  $m\in\N_0$, with $\Omega\su \R^N$ any open set, all of which are isomorphic  to  countable products of Banach spaces. We refer the reader to the survey paper \cite{MM} for further information. Under the assumptions of Lemma \ref{l.spectra} the Fr\'echet space $X$ there is necessarily a quojection. The same is true in Lemma \ref{l.convergenza} and Lemma \ref{l.unif-mean} to follow.

\begin{lemma}\label{l.convergenza} Let $X$ be a  Fr\'echet space and $\{S_n\}_{n=1}^\infty\su\cL(X)$. Suppose that $X=\proj_{j\in\N}(X_j, Q_{j,j+1})$, with $X_j$ a Banach space (having norm $\|\ \|_j$) and linking maps $Q_{j,j+1}\in \cL(X_{j+1},X_j)$ which are surjective for all $j\in\N$, and suppose, for each $j,\ n\in\N$, that  there exists $S_n^{(j)}\in \cL(X_j)$ satisfying   
\begin{equation}\label{e.resol-3}
S_n^{(j)}Q_j=Q_jS_n,
\end{equation}
where $Q_j\in \cL(X,X_j)$, $j\in\N$, denotes the canonical projection of $X$ onto $X_j$ (i.e., $Q_{j,j+1}\circ Q_{j+1}=Q_j$).
Then the following statements are equivalent.
\begin{itemize}
\item[\rm (i)] The limit  $\tau_b$-$\lim_{n\to\infty}S_n=:S$ exists in $\cL_b(X)$.
\item[\rm (ii)] For each $j\in\N$,  the limit $\tau_b$-$\lim_{n\to\infty}S_n^{(j)}=:S^{(j)}$  exists in $\cL_b(X_j)$.
\end{itemize}
In this case, the operators $S\in\cL(X)$ and $S^{(j)}\in \cL(X_j)$, for $j\in\N$, satisfy
\begin{equation}\label{e.resol-op}
Sx=(S^{(j)}x_j)_{j}, \quad x=(x_j)_{j}\in X.
\end{equation}
\end{lemma}

\begin{proof} For each $j\in\N$,  define $q_j(x):=\|Q_jx\|_j$ for $x\in X$. Then $\{q_j\}_{j=1}^\infty\su\G_X$ is a fundamental sequence  generating the lc--topology of $X$ (as $X=\proj_{j\in\N}(X_j, Q_{j,j+1})$). 

(i)$\Rightarrow$(ii). The existence in $\cL_b(X)$ of the stated limit $S\in \cL(X)$ ensures  the existence (in the norm of $X_j$) of  
\begin{equation}\label{e.limite}
\lim_{n\to\infty}S_n^{(j)}Q_jx=\lim_{n\to\infty}Q_jS_nx=Q_jSx,
\end{equation}
for all $j\in\N$ and $x\in X$, via the continuity of  $Q_j$ and \eqref{e.resol-3}. In fact, the weaker requirement that $S_n\to S$ in $\cL_s(X)$ suffices for this.

Fix $j\in\N$. Define $S^{(j)}$ on $X_j=Q_j(X)$ by $S^{(j)}(Q_jx):=Q_jSx$, for  $x\in X$. Then $S^{(j)}\in \cL(X_j)$. Indeed, $S^{(j)}$ is well defined because if $Q_jx=Q_jx'$ for some $x, \ x'\in X$, then $Q_j(x-x')=0$ and so, via \eqref{e.resol-3},  $0=S_n^{(j)}Q_j(x-x')=Q_jS_n(x-x')$ for all $n\in\N$. Passing to the limit for $n\to\infty$, it follows that $0=Q_jS(x-x')$, i.e., $Q_jSx=Q_jSx'$. Clearly, $S^{(j)}$ is linear as both  $Q_j$ and $S$ are linear. Finally, since  $S^{(j)}u=\lim_{n\to\infty}S_n^{(j)}u$, for each $u\in X_j$  (c.f. \eqref{e.limite}) and $\{S_n^{(j)}\}_{n=1}^\infty\su \cL(X_j)$ with $X_j$ a Banach space, it follows from the Uniform Boundedness  Principle that $S^{(j)}$ is continuous and hence, $S_n^{(j)}\to S^{(j)}$ in $\cL_s(X_j)$ for $n\to\infty$. It is routine (via \eqref{e.resol-3}) to check that $S^{(j)}Q_j=Q_jS$.

As noted above,  $X$ is necessarily  a quojection and so there exists $B\in \cB(X)$ 
%a bounded set $B\su X$ 
such that $\mathcal{U}_j\su Q_j(B)$, \cite[Proposition 1]{DiZa}, where $\mathcal{U}_j$ is the closed unit ball of $X_j$. So, by \eqref{e.resol-3} we have
\begin{eqnarray*}
&  &\sup_{u\in \mathcal{U}_j}\|(S_n^{(j)}-S^{(j)})u\|_j\leq \sup_{x\in B}\|(S_n^{(j)}-S^{(j)})Q_j{x}\|_j\\
& &=\sup_{x\in B}\|(Q_j(S_n-S)x\|_j=\sup_{x\in B}{q}_j((S_n-S)x)
\end{eqnarray*}
for all $n\in\N$. Since  $\sup_{x\in B}{q}_j((S_n-S)x)\to 0$ for $n\to\infty$ (by assumption), it follows that $\sup_{u\in \mathcal{U}_j}\|(S_n^{(j)}-S^{(j)})u\|_j\to 0$ for $n\to\infty$, i.e., $\tau_b$-$\lim_{n\to\infty}S_n^{(j)}=S^{(j)}$. 
Since $j\in\N$ is arbitrary, the proof is complete.

(ii)$\Rightarrow$(i). Fix $x=(x_j)_{j}\in X=\proj_{j\in\N}(X_j,Q_{j,j+1})$ and set $Sx:=(S^{(j)}x_j)_{j}$. Then $Sx\in X$. Indeed, $Q_jx=x_j$ for all $j\in\N$ and so, via \eqref{e.resol-3}, we have
\begin{eqnarray*}
&  &Q_{j,j+1}S_{j+1}x_{j+1}=\lim_{n\to\infty}Q_{j,j+1}S_n^{(j+1)}Q_{j+1}x
=\lim_{n\to\infty}Q_{j,j+1}Q_{j+1}S_nx\\
& &\qquad =\lim_{n\to\infty}Q_{j}S_nx=\lim_{n\to\infty}S_n^{(j)}Q_jx=S^{(j)}x_j,
\end{eqnarray*}
for all $j\in\N$, i.e., $Sx\in X$. Clearly, the linearity of the $S^{(j)}$'s imply the linearity of the map $S\colon x\mapsto Sx$, for $x\in X$. Moreover, the continuity of $S$ is a consequence of $X=\proj_{j\in\N}(X_j,Q_{j,j+1})$. Next, fix $j\in\N$ and $B\in \cB(X)$.
% $B\su X$ a bounded set. 
Again  via \eqref{e.resol-3} we have
\begin{eqnarray*}
& & \sup_{x\in B}q_j((S_n-S)x)=\sup_{x\in B}\|(Q_j(S_n-S)x\|_j=\sup_{x\in B}\|(S_n^{(j)}-S^{(j)})Q_jx\|_j\\
& &\qquad = \sup_{u\in Q_j(B)}\|(S_n^{(j)}-S^{(j)})u\|_j
\end{eqnarray*}
for all $n\in\N$. Since $Q_j(B)\in\cB(X_j)$, it follows from the assumption (ii) that  $\sup_{u\in Q_j(B)}\|(S_n^{(j)}-S^{(j)})u\|_j \to 0$ for $n\to\infty$. Accordingly, for each   $j\in\N$ and each $B\in \cB(X)$ 
%bounded set  $B\su X$
 we have $\lim_{n\to\infty}\sup_{x\in B}q_j((S_n-S)x)= 0$, i.e., (i) holds.
%\begin{equation}\label{e.conv-u}
%\lim_{n\to\infty}\sup_{x\in B}q_j((S_n-S)x)= 0,
%\end{equation}
%i.e., (i) holds.
\end{proof}

\begin{remark}\label{r.convergenza}\rm
A careful examination of the proof of Lemma \ref{l.convergenza} shows that the equivalence (i)$\Leftrightarrow$(ii) remains valid if $\tau_b$ is replaced with $\tau_s$.
\end{remark}

\begin{lemma}\label{l.unif-mean} Let $X=\proj_{j\in\N}(X_j, Q_{j.j+1})$ be a Fr\'echet space and operators $S\in \cL(X)$ and $S_j\in \cL(X_j)$, for $j\in\N$, be given which satisfy the assumptions of Lemma \ref{l.spectra} (with $Q_j\in \cL(X,X_j)$, $j\in\N$, denoting the canonical projection of $X$ onto $X_j$ and $\|\ \|_j$ being the norm in the Banach space $X_j$). 
\begin{itemize}
\item[\rm (i)] $S\in \cL(X)$ is power bounded if and only if each $S_j\in \cL(X_j)$, $j\in\N$, is power bounded. 
\item[\rm (ii)] $S\in \cL(X)$ is  uniformly mean ergodic if and only if each $S_j\in \cL(X_j)$, $j\in\N$, is uniformly mean ergodic.
\item[\rm (iii)] $S\in \cL(X)$ is mean ergodic if and only if each $S_j\in \cL(X_j)$, $j\in\N$, is mean ergodic.
\end{itemize}
\end{lemma}

\begin{proof} Let $\{q_j\}_{j=1}^\infty\su \G_X$ be the fundamental sequence of seminorms generating the lc-topology of $X$ as given in the proof of Lemma \ref{l.convergenza}.

(i) Suppose that each $S_j\in \cL(X_j)$, $j\in\N$, is power bounded, i.e., there exists $M_j>0$ such that
\[
\| S_j^nu\|_j\leq M_j\|u\|_j,\quad u\in X_j,\ n\in\N.
\] 
It follows from \eqref{e.resol-1} that $S^n_jQ_j=Q_jS^n$ for all $j,\ n\in\N$. Fix $j\in\N$. Then, for each $n\in\N$ and $x\in X$ we have
\[
q_j(S^nx)=\| Q_jS^n x\|_j=\| S_j^nQ_j x\|_j\leq M_j\|Q_jx\|_j=M_jq_j(x).
\]
Since $\{q_j\}_{j=1}^\infty$ generate the lc-topology of the Fr\'echet space $X$, it follows that $\{S^n\}_{n=1}^\infty\su \cL(X)$ is equicontinuous, i.e., $S$ is power bounded.

Conversely, suppose that  $S$ is power bounded. Fix $j\in\N$ and let $\mathcal{U}_j$ be the closed unit ball of $X_j$. Since $X$ is a quojection, there exists $B\in \cB(X)$ 
 %bounded set $B\su X$
 with $\mathcal{U}_j\su Q_j(B)$. Moreover, the power boundedness of $S$ implies that $C:=\cup_{n\in\N}S^n(B)\in \cB(X)$ 
%is a bounded subset of $X$ 
and hence, there exists $M>0$ such that $q_j(z)\leq M$ for every $z\in C$. Let $u\in \mathcal{U}_j$. Then $u=Q_jx$ for some $x\in B$ and so
\[
\|S_j^nu\|_j=\|S_j^nQ_jx\|_j=\|Q_jS^nx\|_j=q_j(S^nx)\leq M,
\]
for every $n\in\N$. This implies that the operator norms satisfy
$\|S_j^n\|_{op}\leq M$, for $n\in\N$.
Accordingly, $S_j\in \cL(X_j)$ is power bounded.

(ii) For each $n\in\N$ define $\tilde{S}_n:=S_{[n]}\in \cL(X)$ and $\tilde{S}_n^{(j)}:=(S_j)_{[n]}\in \cL(X_j)$, for $j\in\N$. It follows from \eqref{e.resol-1} that 
$\tilde{S}_n^{(j)}Q_j=Q_j\tilde{S}_n$, for $j,\ n\in\N$.
Accordingly, we can apply Lemma \ref{l.convergenza} (with $\tilde{S}_n$ in place of $S_n$ and $\tilde{S}_n^{(j)}$ in place of $S_n^{(j)}$) to conclude that $S$ is uniformly mean ergodic if and only if each $S_j$, $j\in\N$, is uniformly mean ergodic.

(iii) Apply the same argument as in part (ii) but now apply Lemma  \ref{l.convergenza} with $\tau_s$ in place of $\tau_b$; see Remark \ref{r.convergenza}.
\end{proof}
%%%%%%%%%%%%%%%%%%%%%%%%%%%%%%%%%%%%%%%%%%%%%%%%%%%%%%%%%%%%%%%%%%%%%%%%%%%%%%%%%%%%%%%%%%%%%%%%%%%%%%%%%%%%%%%%%%%%%%%%%%%%%%%%%%%%%%%%%%%%%%%%%%%%%%%%%%%%%%%%%%%%%%%%%%%%%%%%%%%%%%%%%%%%%%%%%%%%%%
\textbf{Acknowledgements.} The research of the first two authors was partially 
supported
by the projects MTM2010-15200 and GVA Prometeo II/2013/013 (Spain). 
The second
author gratefully acknowledges the support of the Alexander von Humboldt 
Foundation.

\bigskip
\bibliographystyle{plain}

\end{document}